\documentclass[a4paper,11pt]{article}
\usepackage[utf8]{inputenc}
\usepackage{amsmath}
\usepackage{amsfonts}
\usepackage{amssymb}
\usepackage{amsthm}
\usepackage[english]{babel}
\usepackage{fontenc}
\usepackage{pgfplots}
\pgfplotsset{compat=1.16}
\usepackage{braket}
\usepackage{caption}
\usepackage{caption}
\usepackage{subcaption}
\usepackage{enumitem}
\usepackage{units}
\usepackage{mathtools}
\usepackage{accents}
\usepackage{algorithm}
\usepackage{algorithmic}
\usepackage{mathrsfs}
\usepackage{hyperref}
\mathtoolsset{showonlyrefs}
\usepackage{dsfont}
\usepackage[numbers,sort]{natbib}
\usepackage[hmargin=2.5cm,vmargin=2.5cm,paper=a4paper]{geometry}
\newcommand{\R}{\mathbb R}

\newcommand{\sign}{\operatorname{sign}}

\newcommand{\ic}{\,\Box\,}

\newcommand{\eps}{\varepsilon}

\newcommand{\BV}{\mathrm{BV}}

\newcommand{\dd}{\, \mathrm{d}}
\newcommand{\TV}{\operatorname{TV}}

\newcommand{\TGV}{\operatorname{TGV}}
\newcommand{\ext}{\operatorname{Ext}}

\newcommand{\Lc}{\mathcal{L}}
\newcommand{\Cc}{\mathcal{C}}

\newcommand\restr[2]{{
\left.\kern-\nulldelimiterspace #1 \vphantom{\big|} \right|_{#2} 
}}

\DeclareMathOperator*{\argmin}{arg\,min}
\makeatletter
\newcommand*\bigcdot{\mathpalette\bigcdot@{.5}}
\newcommand*\bigcdot@[2]{\mathbin{\vcenter{\hbox{\scalebox{#2}{$\m@th#1\bullet$}}}}}
\makeatother

\newtheorem{theorem}{Theorem}
\newtheorem{prop}{Proposition}
\newtheorem{ass}{Assumption}

\newtheorem{lemma}{Lemma}
\newtheorem{coroll}{Corollary}
\theoremstyle{definition}
\newtheorem{definition}{Definition}
\newtheorem{remark}{Remark}

\usepackage{comment}
\numberwithin{equation}{section}

\begin{document}
\title{Extremal points of total generalized variation balls in 1D: characterization and applications}
\author{Jos\'e A. Iglesias, Daniel Walter\thanks{Johann Radon Institute for Computational and Applied Mathematics (RICAM), Austrian Academy of Sciences, Linz, Austria (\texttt{jose.iglesias{@}ricam.oeaw.ac.at, daniel.walter@oeaw.ac.at)}}}
\date{}
\maketitle
\begin{abstract}
The total generalized variation (TGV) is a popular regularizer in inverse problems and imaging combining discontinuous solutions and higher order smoothing. In particular, empirical observations suggest that its order two version strongly favors piecewise affine functions. In the present manuscript, we formalize this statement for the one-dimensional TGV-functional by characterizing the extremal points of its sublevel sets with respect to a suitable quotient space topology. These results imply that 1D TGV-regularized linear inverse problems with finite dimensional observations admit piecewise affine minimizers. As further applications of this characterization we include precise first-order necessary optimality conditions without requiring convexity of the fidelity term, and a simple solution algorithm for TGV-regularized minimization problems.
\end{abstract}

The total variation has been a widely considered prior for regularization of inverse problems for several decades. A particularly popular application is denoising of natural images since the seminal work \cite{RudOshFat92}. However, it is arguably even better suited to more involved situations where nearly piecewise constant solutions are expected, like recovery of material inclusions in various physical inverse problems (see for example \cite{FreClaSch10, DonGoeTor15}). The reasons for its success are varied. On the one hand, it typically produces solutions with discontinuities, which can be thought to correspond to edges of objects in the recovered images, or sharp changes of material parameters. On the other, its formulation is purely local and in terms of derivatives, which in particular enables a great amount of rigorous analysis of the resulting solutions. 
One of the main drawbacks of the total variation, however, is the effect generally known as staircasing. Inputs which are not piecewise constant but otherwise very smooth tend to produce solutions with many discontinuities of small amplitude, particularly in the presence of significant noise. A popular cure for this problem is the use of regularizers involving higher order derivatives, which for example can be made to not penalize linear functions. Among such approaches, arguably the most successful is the total generalized variation (TGV) first introduced in \cite{BreKunPoc10}. 

A clearly desirable feature of TGV is that it manages to introduce higher order regularization and avoid staircasing while still allowing sharp discontinuities in the solutions. Also of paramount importance is that, unlike other models satisfying these first requirements, it admits a relatively straightforward predual formulation. In turn, this allows the use of convex nonsmooth optimization tools for its numerical solution, like the ubiquitous \cite{ChaPoc11} which was published around the same time. On the analysis side, topics like existence of minimizers and relations to regularization theory are well explored, see \cite{BreVal11, BreHol14, BreHol20}. However, more detailed characterizations of the behavior of solutions are significantly more scarce and with few exceptions \cite{Val17} they treat the one dimensional case and denoising only \cite{BreKunVal13, PapBre15, PoeSch15}.

In this article, we consider the one-dimensional~$\TGV$-functional
\begin{align*}
\operatorname{TGV}(u) = \inf_{w } \alpha \|Du-w\|_M + \beta \|Dw\|_M
\end{align*}
on an interval~$(0,T)$ from a convex geometry point of view. Here,~$\alpha, \beta >0$ are suitable regularization parameters and~$\|\cdot\|_M$ is the Radon norm on~$(0,T)$, i.e., the dual to the uniform norm. Denoting by~$\Lc$ the space of affine linear functions, we show that the set of equivalence classes
\begin{align*}
\mathcal{B}= \left\{\,u\;|\;\TGV(u)\leq 1\,\right\}+\mathcal{L}
\end{align*} 
is compact with respect to the quotient space topology and can thus be written as the closure of the convex hull of its extremal points~$\operatorname{Ext} \mathcal{B}$, by the Krein-Milman theorem.
The main result of the manuscript gives a precise characterization of~$\ext \mathcal{B}$ as the equivalence classes corresponding to two distinct sets of functions. More in detail, we have 
\begin{align*}
\ext \mathcal{B} =\big(\left\{\,u\;|\; Du=\pm \delta_x/\alpha,~x \in(0,T) \,\right\}  \cup \left\{\,u\;|\; D^2u=\pm \delta_x/\beta,~\operatorname{dist}(x, \partial (0,T))>\beta/\alpha \,\right\} \big)+\mathcal{L}.\end{align*}
Most surprisingly, $\ext \mathcal{B}$ contains the equivalence classes of~\textit{all} piecewise constant functions with a single jump, i.e.,
\begin{align*}
\left\{\,u\;|\; Du=\pm \delta_x/\alpha,~x \in(0,T) \,\right\}+\mathcal{L}
\end{align*}   
while those corresponding to piecewise affine ones are only included if the corresponding kink has a large enough distance to the boundary,
\begin{align*}
\left\{\,u\;|\; D^2u=\pm \delta_x/\beta,~\operatorname{dist}(x, \partial (0,T))>\beta/\alpha \,\right\}+\Lc.
\end{align*}
This characterization has consequences in both the analysis and optimization aspects related to the~$\TGV$-regularizer. Within the first, we derive precise necessary first-order optimality conditions and sufficient conditions to ensure the existence of~\textit{sparse}, i.e.,~piecewise affine, solutions to minimization problems incorporating the~$\TGV$-functional. In particular, based on convex representer theorems, we show that linear inverse problems  with finite-dimensional observation and~$\TGV$-regularization admit such minimizers. These results are not constrained to denoising problems and confirm the intuition that TGV-regularized solutions are encouraged to be piecewise affine. 

On the optimization side, this characterization enables us to formulate and analyze a minimization algorithm that iteratively approximates solutions to~\eqref{def:minprob} by conic combinations of representatives of extremal points. This means that every iterate is piecewise affine. In every iteration, the method alternates between ``adding'' a new extremal point to the current iterate and updating the nonnegative weights in the conic combination via solving a finite dimensional problem with~$\ell_1$-regularization. In particular, the proposed method completely avoids, possibly complicated, evaluations of the~$\TGV$-functional. Moreover, its practical realization does not require a discretization of the underlying interval, meaning that the positions for the discontinuities of the function and its first derivative are found directly and not constrained to be on a discrete grid. We give a preliminary convergence analysis for the method showing subsequential convergence of iterates and a convergence rate in terms of an adequate measure of non-stationarity which dominates the suboptimality of the iterates in the convex case.

As mentioned, at present we are only able to treat the one-dimensional case. In the same spirit as previous works, this serves as a simplified setting in which to infer characteristics of the solutions with the hope that these could be generalized to higher dimensions. However, the one-dimensional case also finds applications in its own right, like in \cite{ReyLoa19} where it is applied for data assimilation for the Burgers equation, whose solutions may contain shocks. Moreover, in a number of works in signal processing \cite{Sel15, OngJac15, KurYamYam18, KitConHir19} priors favoring discontinuous piecewise affine functions are considered for time-domain signals. This motivates the numerical example in our last section. In it, our solution algorithm is applied to the inverse problem of reconstructing a piecewise affine signal from the values of its Fourier transform at a small number of frequencies. This is an archetypical example of a regularization problem with few measurements where a very sparse solution exists in light of our theory, but with the added challenge that the measurements made and the features to be reconstructed are of very different nature.

\paragraph*{Organization of the paper.} In Section \ref{sec:framework}, after recalling the definition of the TGV-functional, we detail how it can be considered in an adequate quotient space whose predual we characterize. In Section \ref{sec:extremals} we prove the announced characterization of extremal points. Section \ref{sec:consequences} deals with necessary first-order optimality conditions and existence of sparse minimizers. In Section \ref{sec:algorithm} we present the minimization algorithm based on extremal points, which is then applied in Section \ref{sec:deconvol} to the test case of reconstruction from few Fourier measurements.

\section{Functional space framework}\label{sec:framework}
Throughout the paper, we work with the total generalized variation of order two with parameters $\alpha, \beta>0$, which is defined as
\begin{equation}\label{eq:tgvsupdef}\TGV_{\alpha, \beta}(u):= \sup \left\{ \int_0^T u\, v'' \, \middle\vert \, v \in C^{\infty}_c(0,T), \, \|v\|_C \leq \alpha, \, \|v'\|_C \leq \beta\right\},\end{equation}
 where $\|\cdot\|_C$ denotes the uniform norm on the space $C_0(0,T)$ of continuous functions vanishing at the boundary. Whenever possible we omit the dependence on the parameters in the notation of $\TGV$ for the sake of brevity. Moreover, according to~\cite[Thm.~3.1]{BreVal11}, we have
\begin{equation}\label{eq:tgvinfdef}
\operatorname{TGV}(u) = \inf_{w \in \mathrm{BV}(0,T)} \alpha \|Du-w\|_M + \beta \|Dw\|_M ,\end{equation}
where by~$\BV(0,T)$ we refer to the space of functions of bounded variation on the interval $(0,T)$ and~$\|\cdot\|_M$ denotes the canonical norm on the space of Radon measures~$M(0,T)=(C_0(0,T))^\ast$. Note that~$\BV(0,T)$ can be seen as the space of $L^1$ functions~$u$ whose distributional derivatives~$Du$ are elements of $M(0,T)$. Abbreviating~$\TV(u)\coloneqq \|Du\|_M$, the corresponding norm is given by
\begin{align*}
\|u\|_{\operatorname{BV}}=\|u\|_{L^1}+ \TV(u) \quad \text{for all}~u \in \BV(0,T).
\end{align*}
For further information on these spaces we refer the reader to Chapters 1 and 3 of \cite{AmbFusPal00}. 

%\begin{equation}\label{eq:tgvinfdef}\begin{aligned}\inf_{u=u_0 + z} \alpha \mathrm{TV}(u_0) + \beta \mathrm{TV}(z') &=  \inf_{z \in \mathrm{BV}^2} \alpha |Du-z'| + \beta |D(z') | = \operatorname{TGV}(u) &= \inf_{w \in \mathrm{BV}} \alpha |Du-w| + \beta |Dw|.\end{aligned}\end{equation}

Directly from the definition \eqref{eq:tgvsupdef}, we see that for all $c, d \in \R$ one has $\TGV(u + cx + d) = \TGV(u)$. Accordingly, we denote the subspaces of constant and affine linear functions by
\begin{align*}
\Cc &\coloneqq \left\{\,u \in \mathrm{BV}(0,T)\;\middle\vert\;u= c\cdot 1_{(0,T)},~c \in \R \,\right\}, \text{ and }\\
\Lc &\coloneqq \big\{\,u \in \mathrm{BV}(0,T) \;\big\vert\; u(x)=cx+d, \, c,d \in \R\big\} = \big\{\,u \in \mathrm{BV}(0,T) \;\big\vert\; Du \in \Cc\,\big\},
\end{align*}
so that $\TGV(u + \ell) = \TGV(u)$ for all $u \in \BV(0,T)$ and $\ell \in \Lc$. As we will see it is beneficial to consider $\TGV$ on the corresponding quotient space~$\BV(0,T)/\Lc$:
\begin{prop}\label{prop:quotientstuff} The total generalized variation $\TGV$ is an equivalent norm on the quotient space $\BV(0,T)/\Lc$, which admits the dual representation
\begin{equation}\label{eq:predual}\begin{gathered}\BV(0,T)/\Lc = \big(C_{0,a}(0,T)\big)^\ast, \text{ with }C_{0,a}(0,T):=\left\{ p \in C_0(0,T) \; \middle\vert \; \int_0^T p \dd x = 0 \right\} \text{ and }\\\langle u + \Lc, p\rangle_{\big(BV(0,T)/\Lc\,,\, C_{0,a}(0,T)\big)} = \int_{(0,T)} p \dd Du.\end{gathered}\end{equation}
In particular, the unit ball 
\begin{equation}\label{eq:unitball} \mathcal{B} := \big\{ u + \Lc \in \BV(0,T)/\Lc \;\big\vert\; \TGV(u) \leq 1 \big\}\end{equation}
is sequentially compact with respect to the corresponding weak-* topology in $\BV(0,T)/\Lc$.
\end{prop}
\begin{proof}
To begin with, we notice that both $\Cc$ and $\Lc$ are closed subspaces of $\BV(0,T)$ with respect to the weak-* topology. Moreover, we may write
\[\BV(0,T)/\Lc = \big( \BV(0,T)/\Cc \big) / ( \Lc / \Cc).\]
In the one dimensional situation we consider, the map
\begin{align*}
\Lambda:BV(0,T)/\Cc &\to M(0,T)\\
u + \Cc &\mapsto Du
\end{align*}
is clearly continuous by the definition of the corresponding norms, and has as an inverse the map assigning to each $\mu \in M(0,T)$ the equivalence class $\big[x \mapsto \mu\big((0,x)\big)\big]+\Cc$. Using the bounded inverse theorem, this implies that it is an isomorphism of Banach spaces. Furthermore, the image $\Lambda( \Lc / \Cc)=\R \mathscr{L}$ is the set of scalar multiples of the Lebesgue measure restricted to $(0,T)$, which we denote by $\mathscr{L}$. By results for general Banach spaces (an application of Hahn-Banach, essentially) we have \cite[Prop.~2.7]{FabEtAl01} that if $C_{0,a}(0,T)^\perp = \R \mathscr{L}$ then $M(0,T)/(\R \mathscr{L}) = \big(C_{0,a}(0,T)\big)^\ast$. To see the equality for the annihilator, consider 
\[\mu \in M(0,T) \text{ satisfying } \int_{(0,T)} p \dd \mu = 0 \text{ for all }p \in C_{0,a}(0,T),\]
and fix an arbitrary function $q_0 \in C_0(0,T)$ with $\int_0^T q_0 \dd x = 1$. Now, for each $q \in C_0(0,T)$, just consider the function
\[C_{0,a}(0,T) \ni p = q - \left(\int_0^T q \dd x\right) \, q_0, \text{ so that }\int_{(0,T)} q \dd \mu = \left(\int_0^T q \dd x\right) \left(\int_{(0,T)} q_0 \dd \mu\right).\]
Since the number $\int_{(0,T)} q_0 \dd \mu$ is independent of $q \in C_0(0,T)$ and by the Riesz representation theorem $M(0,T)=\big(C_0(0,T)\big)^\ast$, this means that
\[\mu = \left( \int_{(0,T)} q_0 \dd \mu \right)\mathscr{L}.\]
To see that $\TGV$ is an equivalent norm on $\BV(0,T)/\Lc$, fix~$u \in \BV(0,T)$ as well as~$w_u$ with
\begin{align*}
\inf_{w \in \mathrm{BV}} \alpha \|Du-w\|_M + \beta \|Dw\|_M 
&=\alpha \|Du-w_u\|_M + \beta \|Dw_u\|_M.
\end{align*}
Now, we can first use the Poincar\'e-Wirtinger inequality \cite[Thm.~3.44, Rem.~3.45]{AmbFusPal00} in $\BV(0,T)$ twice to obtain that
\begin{equation}\label{eq:tgvlowerbound}\begin{aligned}
\TGV(u) &= \inf_{w \in \mathrm{BV}} \alpha \|Du-w\|_M + \beta \|Dw\|_M \\
&=\alpha \|Du-w_u\|_M + \beta \|Dw_u\|_M\\
&\geq \alpha \|Du-w_u\|_M + CT\beta \big\|w_u- T^{-1}\textstyle{\int_0^T}w_u \big\|_{L^1}\\
&\geq \min(\alpha, CT\beta) \big( \|Du-w_u\|_M + \big\|w_u- T^{-1}\textstyle{\int_0^T}w_u \big\|_{L^1} \big)\\
&\geq \min(\alpha, CT\beta) \inf_{c \in \R} \big( \|Du-w_u\|_M + \big\|w_u-c\big\|_{L^1} \big)\\
&\geq \min(\alpha, CT\beta) \inf_{c \in \R} \|D(u-cx)\|_M\\
&\geq CT \min(\alpha, CT\beta) \inf_{c \in \R} \big\|u-cx-T^{-1}\textstyle{\int_0^T}(u-cx)\big\|_{L^1}\\
&\geq CT \min(\alpha, CT\beta) \inf_{c,d \in \R} \big\|u-cx-d\big\|_{L^1},\\
\end{aligned}\end{equation}
%&= \min(\alpha, C\beta) \inf_{\ell \in \Lc} \big\|u-\ell\big\|_{L^1}.\\
with which we can estimate
\begin{align*}\|u+\Lc\|_{\BV/\Lc} &= \inf_{\ell \in \Lc} \|u-\ell\|_{\BV} = \inf_{c,d \in \R} \|u-cx-d\|_{\BV} \\
&= \inf_{c,d \in \R} \|u-cx-d\|_{L^1} + \|D(u-cx)\|_M \leq \frac{2}{CT \min(\alpha, CT\beta)} \TGV(u).\end{align*}
On the other hand, testing the definition of $\TGV$ with constant $w \equiv c$ tells us that
\[\TGV(u) \leq \inf_{c \in \R} \alpha \|Du-c\|_M = \alpha \inf_{c \in \R} \|D(u-cx)\|_M = \inf_{\ell \in \Lc} \|D(u-\ell)\|_M \leq \|u+\Lc\|_{\BV/\Lc}.\]
We remark that this equivalence of norms is not specific to the one-dimensional case, and in fact a similar estimate is proved in \cite[Thm.~3.3]{BreVal11}.

With the equivalence of norms and having identified a separable predual space, compactness of $B$ can be characterized by sequences and follows by an application of the Banach-Alaoglu theorem.
\end{proof}

\section{Extremal points of~\texorpdfstring{$\TGV$}{TGV}}\label{sec:extremals}
In this section, we consider the extremal points of the set~$\mathcal{B}$:
\begin{definition}
An element~$u \in \mathcal{B}$ is called an extremal point of~$\mathcal{B}$ if
\begin{align*}
u=u_1+s(u_2-u_1)\ \text{for some}~u_1,u_2 \in \mathcal{B}\text{ and }s \in (0,1) \text{ implies } u=u_1=u_2.
\end{align*}
\end{definition} 
Since~$\mathcal{B}$ is weak* compact in~$\BV(0,T)/\mathcal{L}$, the set of all extremal points of~$\mathcal{B}$, denoted by~$\ext \mathcal{B}$ is nonempty due to the Krein-Milman theorem.
We prove the following characterization:
\begin{theorem}\label{thm:tgvextremals}
The extremal points of $\mathcal{B}= \big\{ u + \Lc \in \BV(0,T)/\Lc \;\big\vert\; \TGV(u) \leq 1 \big\}$ satisfy
\begin{equation}\label{eq:maybeextremals}\ext \mathcal{B} = \pm \left( \left\{ \frac{1}{\alpha}S_{x_0} \;\Big\vert\;  x_0 \in (0,T)\right\} \cup \left\{ \frac{1}{\beta}K_{x_1} \;\Big\vert\;  \operatorname{dist}(x_1, \partial(0,T)) > \frac{\beta}{\alpha}\right\} \right ) + \Lc,\end{equation} where, denoting by $1_{(a,b)}$ the characteristic function of an interval $(a,b)$,
\[S_{x_0}=1_{(x_0, T)}\text{ and }K_{x_1}=\begin{cases}1_{(0, x_1)} (x_1-x) &\text{ if }x_1 < T/2 \\ 1_{(x_1, T)} (x-x_1) &\text{ if }x_1 \geq T/2.\end{cases}\]
\end{theorem}

The strategy is centered around the fact that since in one dimension we can always find primitives, $\TGV$ can be rewritten as an infimal convolution:
\begin{lemma}\label{lem:tgvinfconv} For every~$\bar{u} \in \BV(0,T)$ we have
\begin{equation}\label{eq:tgvinfconv}
\operatorname{TGV}(\bar{u})= \inf_{\substack{v,h \in \mathrm{BV}(0,T), \\ \bar{u}=v+h}} \left \lbrack\, \alpha \operatorname{TV}(v)+ \beta \operatorname{TV}(Dh) \,\right \rbrack,
\end{equation}
and if~$\bar{w} \in \BV(0,T)$ is such that 
\begin{align*}
\operatorname{TGV}(\bar{u})= \alpha \|D\bar{u}-\bar{w}\|_M+ \beta \|D\bar{w}\|_M
\end{align*}
then the infimum is attained for
\begin{equation}\label{eq:integrateminima}
\bar{v}= \bar{u}- \int^{\boldsymbol{\cdot}}_0 \bar{w}~\mathrm{d} t,~\bar{h}=\int^{\boldsymbol{\cdot}}_0 \bar{w}~\mathrm{d} t.\end{equation}
Moreover we can define $\TGV$ also in the quotient space $\BV(0,T)/\Lc$, and there it can also be written as an infimal convolution:
\begin{equation}\label{eq:tgvinquotient}\TGV(\bar u)=\TGV(\bar u + \Lc)= \inf_{\substack{v+\Lc,h+\Lc \in \mathrm{BV}(0,T)/\Lc, \\ \bar{u}=v+h}} \left \lbrack\, \alpha \inf_{\ell \in \Lc}\TV(v+\ell)+ \beta \operatorname{TV}(Dh) \,\right \rbrack.\end{equation}
\end{lemma}
\begin{proof}
First note that since~$Dh \in \mathrm{BV}(0,T)$ we have, denoting by $D^a$ and $D^s$ the absolutely continuous and singular parts with respect to the Lebesgue measure, that $D^s h=0$ and thus
\begin{align*}
D^s h=0,~D^s \bar{u}=D^s v,~D^a v=D^a \bar{u} -Dh.
\end{align*} 
As a consequence we can rewrite
\begin{align*}
\inf_{\substack{v,h \in \mathrm{BV}(0,T), \\ \bar{u}=v+h}} \left \lbrack\, \alpha \operatorname{TV}(v)+ \beta \operatorname{TV}(Dh) \,\right \rbrack &= \inf_{Dh \in \mathrm{BV}(0,T)} \alpha \|D^s \bar{u}\|_M+ \alpha \|D^a \bar{u}-Dh\|_M+ \beta \operatorname{TV}(Dh)\\ &=  \inf_{w \in \mathrm{BV}(0,T)} \alpha \|D \bar{u}-w\|_M+ \beta \operatorname{TV}(w)
\end{align*}
where we use that for every~$w \in \mathrm{BV}(0,T)$ there is~$W \in \mathrm{BV}(0,T)$ with~$DW=w$. Moreover we have~$\bar{u}=\bar{v}+\bar{h}$ and
\begin{align*}
\alpha \operatorname{TV}(\bar{v})+ \beta \operatorname{TV}(D\bar{h})= \alpha \|D\bar{u}-\bar{w}\|_M+ \beta \|D^2\bar{w}\|_M=\operatorname{TGV}(\bar{u})
\end{align*}
finishing the proof of \eqref{eq:tgvinfconv} and \eqref{eq:integrateminima}. Finally, to deduce \eqref{eq:tgvinquotient} from \eqref{eq:tgvinfconv} we just notice that by definition $\TV(Dh) = \TV\big(D(h+\ell)\big)$ for all $\ell \in \Lc$.
\end{proof}

Then we have the following more general negative result, which allows us to reduce the possible candidates for extremal points:
\begin{lemma}\label{lem:infconvsep}Let $f,g$ be convex positively one-homogeneous functionals defined on some Banach space $X$, and
\begin{equation}\label{eq:infconvfg}(f \ic g)(u):= \inf_{u = v+w} f(v) + g(w)\end{equation}
their infimal convolution. Given $u_0 \in X$, assume that either $f(u_0)\neq 1$ or $g(u_0)\neq 1$, and $v_0, w_0$ are such that 
\[(f \ic g)(u_0)=1, \quad u_0 = v_0 + w_0, \quad f(v_0) + g(w_0)=1, \quad f(v_0)\neq 0, \quad g(w_0) \neq 0\]
so \eqref{eq:infconvfg} is exact at $u_0$ with the infimum attained by $v_0$ and $w_0$. Then $u_0$ cannot be an extremal point of the convex set 
\[\left\{ u \in X \, \middle\vert\, (f \ic g)(u)\leq 1\right\}.\]
\end{lemma}
\begin{proof}
First, notice that we must have $f(v_0) \leq g(v_0)$. To see this, assume for the sake of contradiction that $g(v_0) < f(v_0)$. Being convex and one-homogeneous $g$ is sublinear, so we have that
\[g(v_0+w_0)\leq g(v_0) + g(w_0) < f(v_0)+g(w_0)\]
and the pair $(v_0, w_0)$ was not optimal for the infimal convolution, since $(0, v_0+w_0)$ has lower energy. Similarly we conclude that $g(w_0) \leq f(w_0)$. 

Moreover, we claim that
\begin{equation}\label{eq:goodseparation}(f \ic g)(v_0) = f(v_0) \in (0,1),\quad (f \ic g)(w_0) = g(w_0) \in (0,1).\end{equation}
To see this, notice that if on the contrary we had for some $v_1, w_1$ with $v_0=v_1+w_1$ that
\[(f \ic g)(v_0) = f(v_1) + g(w_1) < f(v_0)\]
we would end up, using the sublinearity, with
\[f(v_0)+g(w_0) > f(v_1)+g(w_0)+g(w_1) \geq f(v_1)+g(w_0+w_1),\]
and again the pair $(v_0, w_0)$ could not have been optimal for $(f \ic g)(u_0)$. The claim for $w_0$ is again entirely similar.

Using in \eqref{eq:goodseparation} that the infimal convolution is one-homogeneous as well we arrive at
\[ (f \ic g)\left(\frac{v_0}{f(v_0)}\right)=1, \quad (f \ic g)\left(\frac{w_0}{g(w_0)}\right)=1,\]
which means that we can write 
\[u_0 = f(v_0)\frac{v_0}{f(v_0)} + g(w_0)\frac{w_0}{g(w_0)}.\]
Now if $v_0/f(v_0)\neq w_0/g(w_0)$, these two points and $u_0$ are collinear in $\left\{ u \in X \, \middle\vert\, (f \ic g)(u)=1\right\}$, and $u_0$ could not be extremal. Otherwise \[v_0=c\, w_0\text{ with }c:=f(v_0)/g(w_0),\text{ so }u_0 = (1+c)w_0=(1+c^{-1})v_0.\] 
By one-homogeneity and the first part of the proof
\begin{align*}
g(u_0)=(1+c)g(w_0) &\leq (1+c)f(w_0)=f(u_0)\text{ and }\\
f(u_0)=(1+c^{-1})f(v_0) &\leq (1+c^{-1})g(v_0)=g(u_0),
\end{align*}
from which we conclude $f(u_0)=g(u_0)=(f \ic g )(u_0)=1$, contrary to our assumption.
\end{proof}

Applied to our case, the above lemma yields:
\begin{prop}\label{prop:extremalunion}Denoting $\mathcal{B} := \{ u + \Lc \;\big\vert\; \TGV(u) \leq 1\}$ we have
\begin{equation}\label{eq:extremalunion}\ext \mathcal{B} \subset \pm \left(\left\{\,u \;\Big\vert\; \|Du\|_M \leq \frac{1}{\alpha}\,\right\} \cup \left\{\,u \;\Big\vert\; \|D^2 u \|_M \leq \frac{1}{\beta}\,\right\}\right) + \Lc \subset \mathcal{B},\end{equation}
and every element in $\ext \mathcal{B}$ is either of the form $\pm \{S_{x}/\alpha + \Lc\}$ or $\pm \{K_{x}/\beta +\Lc\}$.
\end{prop}
\begin{proof}We consider $\TGV$ as an infimal convolution in $\BV(0,T)/\Lc$ as in \eqref{eq:tgvinquotient} of Lemma \ref{lem:tgvinfconv}. First, we notice that for each equivalence class $u + \Lc$ there exist representatives $\bar u$ for which
\begin{equation}\label{eq:goodrepr}\inf_{\ell \in \Lc}\|D(u+\ell)\|_M = \|D\bar u\|_M ,\end{equation}
since this is a convex and coercive functional in $\Lc / \Cc$. This also implies that the functional $u \mapsto \inf_{\ell \in \Lc}\|D(u+\ell)\|_M$ is convex, since if $u+\Lc = (1-\lambda)(u_1+\Lc)+\lambda(u_2 + \Lc)$ and $\bar{u_1}, \bar{u_2}$ are such representatives we have
\begin{align*}\inf_{\ell \in \Lc}\|D(u+\ell)\|_M &\leq \big\|D\big((1-\lambda)\bar{u_1}+\lambda \bar{u_2}\big)\big\|_M \\&\leq (1-\lambda)\|D\bar{u_1}\|_M +\lambda\|D\bar{u_2}\|_M \\&= (1-\lambda)\inf_{\ell \in \Lc}\|D(u_1+\ell)\|_M + \lambda \inf_{\ell \in \Lc}\|D(u_2+\ell)\|_M .\end{align*}
Moreover since $\Lc$ is a linear subspace the functional above is also positively one-homogeneous, so we can apply Lemma \ref{lem:infconvsep}. In view of it, to arrive at \eqref{eq:extremalunion} we just need to check that
\[\left\{ u + \Lc \, \Big\vert \, \alpha \inf_{\ell \in \Lc}\|D(u+\ell)\|_M \leq 1\right \} = \left\{\,u \;\Big\vert\; \|Du\|_M \leq \frac{1}{\alpha}\,\right\} + \Lc,\]
which again follows directly by the existence of minimal representatives \eqref{eq:goodrepr}.

Notice that the assumptions of Lemma \ref{lem:infconvsep} are not a restriction since the case $\|Du\|_M /\alpha = \|D^2 u\|_M /\beta =\TGV(u)= 1$ (for $u$ an optimal representative as above) is also included in \eqref{eq:extremalunion}.
\end{proof}

For the positive direction, the piecewise constant functions in \eqref{eq:maybeextremals} can be handled directly:
\begin{prop}\label{prop:pwcextremals}
Let~$u \in \{S_x / \alpha \}+ \Lc$,~$x \in (0,T)$, be given. Assume that there are~$u_1+\Lc, u_2+\Lc \in \mathcal{B}$ as well as~$\lambda \in (0,1)$ with~$u=(1-\lambda)u_1+ \lambda u_2$. Then we have~$u_1, u_2 \in \{S_x / \alpha \}+ \Lc$.
\end{prop}
\begin{proof}
Again we may assume that~$\mathrm{TGV}(u_1)=\mathrm{TGV}(u_2)=1$. For every~$i=1,2$ there exists~$w_i \in \mathrm{BV}(0,T)$ with
\begin{align*}
 \alpha \|Du_i-w_i\|_M + \beta \|Dw_i\|_M=\alpha \|D^s u_i \|_M +  \alpha \|D^a u_i - w_i \|_M + \beta \|Dw_i \|_M = \operatorname{TGV}(u_i)=1.
\end{align*}
In particular, this implies
\begin{align} \label{eq:estforDu1}
\alpha \|D^s u_i\|_M \leq \alpha \|D^s u_i\|_M +  \alpha \|D^a u_i - w_i\|_M+ \beta \|Dw_i\|_M= \operatorname{TGV}(u_i)=1.
\end{align}
By assumption we have that
\begin{align*}
Du=\delta_x/\alpha + c =(1-\lambda)Du_1+\lambda Du_2
\end{align*}
for some~$c\in \R$. Thus, separating into singular and absolutely continuous parts, we get
\begin{align*}
\delta_x /\alpha=(1-\lambda)D^s u_1+\lambda D^s u_2,~c= (1-\lambda)D^a u_1+\lambda D^a u_2.
\end{align*}
By extremality of~$\delta_x$ for the unit ball in~$M(0,T)$ and~$\lambda \in (0,1)$, the first equality is only possible if~$D^s u_1=D^s u_2= \delta_x /\alpha$. Now, since~$\alpha |D^s u_i |=1$,~\eqref{eq:estforDu1} yields
\begin{align*}
\alpha \|D^a u_i - w_i\|_M+ \beta \|Dw_i\|_M=\inf_{w \in \mathrm{BV}(0,T)} \left \lbrack\, \alpha \|D^a u_i - w\|_M+ \beta \|Dw\|_M \,\right \rbrack=0.
\end{align*}
which is only possible if~$D^a u_i \in \Cc$. Combining the previous observations we get that~$Du_i \in \{ \delta_x /\alpha\}+ \Cc$ and thus~$u_i \in \{S_x / \alpha \}+ \Lc$.
\end{proof}

To proceed further, we need to examine in detail the behaviour of $\TGV$ for piecewise affine functions of the form $K_{x_1}$.

\begin{prop}\label{prop:tgvpwlinear}Consider the functions $K_{x_1}$ defined in Theorem \ref{thm:tgvextremals}. Assuming $x_1 \geq T/2$, if
\begin{equation}\label{eq:x1betaalpha}\frac{\beta}{\alpha} \leq T - x_1\end{equation}
we have $\TGV_{\alpha, \beta}(K_{x_1}) = \beta \|D^2 u\|_M= \beta$, and otherwise $\TGV_{\alpha, \beta}(K_{x_1}) = \alpha \|Du\|_M=\alpha(T-x_1)$. For $x_1 < T/2$ analogous statements hold, in which each appearance of $T - x_1$ is replaced by $x_1$. Moreover, as long as $\beta/\alpha \neq \min(x_1, T - x_1)$ the inner minimization problem
\begin{equation}\label{eq:pwlL1TV}\argmin_{w \in \mathrm{BV}(0,T)} \alpha \|DK_{x_1}-w\|_M + \beta \|Dw\|_M\end{equation}
admits a unique minimizer $w_{K_{x_1}}$, where the uniqueness is understood in terms of almost everywhere equality.
\end{prop}
\begin{proof}By definition, in case $x_1 \in [T/2, T)$ we have $K_{x_1}(x) = (x-x_1)^+$ and if $x_1 \in (0,T)$ instead $K_{x_1}(x) = (x-x_1)^+ - x + x_1$. Since \eqref{eq:pwlL1TV} is invariant by sign changes and constant shifts in $DK_{x_1}$, we can consider without loss of generality just the function $(\cdot-x_1)^+$. We then notice that since $D(\cdot-x_1)^+(x) \in \{0,1\}$, as in \cite[Prop.~4.5]{DuvAujGou09} by comparing with truncations one sees that any minimizer of \eqref{eq:pwlL1TV} takes values in the interval $[0,1]$ almost everywhere. Moreover, by \cite[Thm.~3.27]{AmbFusPal00} we can compute one-dimensional $\TV$ through the essential (pointwise) variation
\begin{equation}\label{eq:pointwisevariation}\TV(w)=\inf_{z = w ~a.e.} \sup \left\{ \sum_{i = 1}^{n-1} |z(t_{i+1})-z(t_i)| \,\middle\vert\, n \geq 2,\ 0 < t_1 < \ldots < t_n < T  \right\},\end{equation}
and the infimum is achieved for a particular representative $z_w$, which will be assumed in the rest of the proof (that is, from here on $w=z_w$). In particular \eqref{eq:pointwisevariation} implies that
\begin{equation}\label{eq:tvinfsup}\TV(w) \geq \sup_{x \in (0,T)} w(x) - \inf_{x \in (0,T)} w(x) = \sup w - \inf w.\end{equation}
To see this, assume $w$ is not constant (since otherwise there is nothing to prove) and consider sequences $\{x_k\}_k$ and $\{y_k\}_k$ for which $w(x_k) \to \inf w$ and $w(y_k) \to \sup w$. Since $w$ is not constant, for all $k$ large enough we must have $\min(x_k, y_k) < \max(x_k, y_k)$, which allows to use these points for partitions in \eqref{eq:pointwisevariation}. With \eqref{eq:tvinfsup} we have for \eqref{eq:pwlL1TV} the lower energy bound
\begin{equation}\label{eq:energybound}\begin{aligned}\alpha \|DK-w\|_M + \beta \|Dw\|_M &\geq \alpha x_1 \inf w + \alpha (T-x_1)(1-\sup w) + \beta (\sup w - \inf w)
\\&=(\alpha x_1 - \beta) \inf w + (\beta - \alpha (T-x_1))\sup w + \alpha(T-x_1),
\end{aligned}\end{equation}
where we note that the last term does not depend on $w$. Our strategy is to first minimize the right hand side with respect to $\inf w$ and $\sup w$ and then produce functions with these extrema for which \eqref{eq:energybound} holds with equality, which must then be minimizers. To accomplish this for every $\alpha$, $\beta$ and $x_1$ we distinguish four cases:
\begin{itemize}
\item If $x_1 \geq \beta/\alpha$ and $(T-x_1) \geq \beta/\alpha$, then the right hand side of \eqref{eq:energybound} is clearly lowest if $\inf w = 0$ and $\sup w = 1$, in which case it reaches the value $\beta$. Moreover, when $w = D(\cdot-x_1)^+ = S_{x_1}$ we have equality in \eqref{eq:energybound}, so it is a minimizer with cost $\beta$.

\item For $x_1 < \beta/\alpha$ and $(T-x_1) < \beta/\alpha$, the first term compels us to maximize $\inf w$ while the second term enforces minimizing $\sup w$. In case $x_1 - \beta/\alpha > (T-x_1)- \beta/\alpha$ (that is, $x_1 > T/2$) the second term is stronger, so it is advantageous to choose $\sup w = 0$, hence $w \equiv 0$ minimizes. In case $x_1 \leq T/2$ we can use $w \equiv 1$ instead. Moreover, in both cases the bound \eqref{eq:energybound} is attained with the value $\alpha$.

\item The two remaining cases $x_1 < \beta/\alpha \leq (T-x_1)$ and $x_1 \geq \beta/\alpha > (T-x_1)$ can be handled as simplifications of the previous one, since in this case there is no competition between the two terms and we can just choose $w \equiv 1$ or $w \equiv 0$ respectively. Again, there is equality in \eqref{eq:energybound} with value $\alpha$.
\end{itemize}
Finally, we notice that the only possibility that does not lead to a unique minimizer is when one of the terms of the right hand side of \eqref{eq:energybound} vanishes. This means that either $x_1 = \beta/\alpha$ in which case $w = c\,D(\cdot-x_1)^+$ is a minimizer for all $c \in [0,1]$, or that $(T-x_1) = \beta/\alpha$ with $w = c+(1-c) D(\cdot-x_1)^+$ minimizing for all $c \in [0,1]$.
\end{proof}

With this, we can prove that all of the piecewise affine functions in \eqref{eq:maybeextremals} are indeed extremal:
\begin{prop}\label{prop:pwlextremals}
Let~$u \in \{K_x / \beta \}+ \Lc$,~$\operatorname{dist}(x, \partial(0,T))> \beta/\alpha$, be given. Assume that there are~$u_1, u_2 \in B$ as well as~$\lambda \in (0,1)$ with~$u=(1-\lambda)u_1+ \lambda u_2$. Then we have~$u_1, u_2 \in \{K_x / \beta \}+ \Lc$.
\end{prop}
\begin{proof}
We may assume without loss of generality that~$\mathrm{TGV}(u_1)=\mathrm{TGV}(u_2)=1$. For~$i=1,2$ there exists~$w_i \in \mathrm{BV}(0,T)$ with
\begin{align*}
\beta \|Dw_i\|_M \leq \alpha \|Du_i-w_i\|_M+ \beta \|Dw_i\|_M= \operatorname{TGV}(u_i)=1.
\end{align*}
Moreover, by Proposition~\ref{prop:tgvpwlinear} the problem
\begin{align*}
\min_{w \in \mathrm{BV}(0,T)} \left \lbrack \,\alpha \|Du-w\|_M+\beta \|Dw\|_M \,\right \rbrack
\end{align*}
has a unique minimizer given by~$w=Du$. Now, by assumption, we have~$Du=(1-\lambda)Du_1+ \lambda Du_2$. Consequently, the function~$w_\lambda=(1-\lambda) w_1+\lambda w_2$ satisfies
\begin{align*}
1=\operatorname{TGV}(u) &\leq\alpha \|Du-w_\lambda\|_M+ \beta \|Dw_\lambda\|_M  \\
& \leq (1-\lambda) \operatorname{TGV}(u_1) + \lambda \operatorname{TGV}(u_2)=1. 
\end{align*}
Thus we conclude~$w_\lambda=Du$ and
\begin{align*}
\delta_x /\beta= (1-\lambda) D w_1 + \lambda D w_2
\end{align*}
Again, by extremality of~$\delta_x $ for the unit ball in~$M(0,T)$, we get~$w_i \in \{S_x / \beta \}+ \Cc$. Now, since
\begin{align*}
1= \beta \|Dw_i\|_M \leq \operatorname{TGV}(u_i)=1
\end{align*}
we arrive at~$w_i=Du_i$ which finally yields~$u_i \in  \{K_x / \beta \}+ \Lc $.
\end{proof}

On the other hand the equality case can be resolved manually:
\begin{remark}\label{rem:equalcasenot} If $\operatorname{dist}(x_1, \partial(0,T)) = \beta/\alpha$ then $K_{x_1}/\beta + \Lc$ is not in $\ext B$. To see this, assuming without loss of generality that $x_1 > T/2$, we can write $K_{x_1}/\beta$ as the convex combination
\[\frac{1}{\beta}K_{x_1} = \frac12 u_1 + \frac12 u_2 \,\text{ with }\, u_1 = \frac{2}{\beta}K_{x_1} - \frac{2}{\beta}K_{(T+x_1)/2}\text{ and }\, u_2 = \frac{2}{\beta}K_{(T+x_1)/2},\]
and $u_1 + \Lc, u_2 + \Lc$ belong to $\mathcal{B}$ since arguing as in Proposition \ref{prop:tgvpwlinear} we have that the inner minimizers must be $w_i\equiv 0$ and
\[\TGV(u_1) = \alpha \|Du_1\|_M= \TGV(u_2)=\alpha \|Du_2\|_M=\frac{\alpha}{\beta}(T-x_1)=1.\]
%\begin{gather*}
%\alpha\left|D\left(\frac{2}{\beta}K_{(1+x_1)/2}\right)\right|(0,T)=\frac{\alpha}{\beta}(1-x_1)=1, \text{ but }\\
%\beta \left|D^2\left(\frac{2}{\beta}K_{x_1} - \frac{2}{\beta}K_{(1+x_1)/2}\right)\right|(0,T)=4 \,\,\text{ and }\,\,\beta\left|D^2\left(\frac{2}{\beta}K_{(1+x_1)/2}\right)\right|(0,T)=2.
%\end{gather*}
\end{remark}

Finally, putting the above facts together we get:
\begin{proof}[Proof of Theorem \ref{thm:tgvextremals}]
By Proposition \ref{prop:extremalunion}, elements in $\ext \mathcal{B}$ can only be equivalence classes corresponding to either piecewise constant functions with one jump in $\pm\{S_x / \alpha \} + \Lc$ or piecewise affine functions with one kink in $\pm\{K_x / \beta \} + \Lc$. Proposition \ref{prop:pwcextremals} implies that all the piecewise constant candidates are indeed extremal. For the piecewise affine candidates, we see by Proposition \ref{prop:pwlextremals} that those with $\operatorname{dist}(x, \partial(0,T))> \beta/\alpha$ are also extremal, and by Proposition \ref{prop:tgvpwlinear} that those with $\operatorname{dist}(x, \partial(0,T)) < \beta/\alpha$ are not. The equality case is resolved by Remark \ref{rem:equalcasenot}.
\end{proof}

\section{Optimization problems with~\texorpdfstring{$\TGV$}{TGV}-regularization} \label{sec:consequences}
In the following, we discuss the consequences of Theorem~\ref{thm:tgvextremals} on minimization problems of the form
\begin{align} \label{def:minprob}
\min_{u \in \BV(0,T)} J(u), \quad J(u)=f(u) + \TGV_{\alpha, \beta}(u)\tag{$\mathcal{P}$}.
\end{align}
Here,~$f \colon L^2(0,T) \to \mathbb{R}$ denotes a smooth but not necessarily convex fidelity term and the~$\TGV$-functional plays the role of a regularizer. From a practical point of view, problems of this form are interesting since the inclusion of the~$\TGV$-functional is expected to favor piecewise affine minimizers. This section aims at formalizing this observation. In particular, we present sufficient conditions for the existence of a solution~$\bar{u}$ which is~\textit{sparse} in the sense that, up to a shift in~$\mathcal{L}$, it can be written as a finite conic combination of piecewise constant and continuous piecewise affine functions, i.e., there are sets~$\{\bar{x}^S_j\}^{N_S}_{j=1},~\{\bar{x}^K_j\}^{N_K}_{j=1}$ in~$(0,T)$ such that
\begin{align}
\label{def:sparseuintro}
\bar{u}= (1/\alpha)\sum^{N_S}_{j=1}\bar{\sigma}^S_j \bar{\lambda}^S_j S_{\bar{x}^S_j}+ (1/\beta)\sum^{N_K}_{i=j} \bar{\sigma}^K_j \bar{\lambda}^K_j K_{\bar{x}^K_j}+ \bar{a}x+\bar{b}.
\end{align}
for some~$\bar{\sigma}^S_j,\bar{\sigma}^K_j \in \{-1,1\}$,~$\bar{\lambda}^S_j,\bar{\lambda}^K_j \geq 0,\bar{a}, \bar{b} \in \R$. Two characteristic cases are discussed. First, we show that the kinks and jumps of a solution~$\bar{u}$ correspond to maximizers and minimizers of certain dual variables. This can be derived from necessary first-order optimality conditions for Problem~\eqref{def:minprob} which are provided in Section~\ref{subsec:subdiff}. Consequently, if these only admit~\textit{finitely many} global extrema, the function~$\bar{u}$ is of the form~\eqref{def:sparseuintro}. Second, we consider structured objective functionals of the form~$f(u)=F(Hu)$ where~$F$ is convex and~$H$ is a linear operator with finite range. In this case, the existence of a sparse solution follows from~\textit{convex representer theorems} for abstract minimization problems, see~\cite{BreCar20,Boy}. Loosely speaking, applied to the present problem, these state the existence of a minimizer~$\bar{u}$ which is given by a finite conic combination of functions~$\bar{u}_j$ with~$\bar{u}_j+ \mathcal{L} \in \ext  \mathcal{B}$. Thus, in virtue of Theorem~\ref{thm:tgvextremals}, it will be of the form~\eqref{def:sparseuintro}. 
\subsection{Linear minimization over~\texorpdfstring{$\TGV$}{TGV}-balls} \label{subsec:linearmin}
For abbreviation set
\begin{align*}
\mathcal{S} \coloneqq \left\{\, \pm S_x/\alpha\;|\;x \in (0,T)\,\right\},~\mathcal{K} \coloneqq \left\{\,\pm K_x/\beta\;|\;x \in (\beta/\alpha,T-\beta/\alpha)\,\right\}
\end{align*}
throughout the following sections. Moreover, with abuse of notation, define
\begin{align*}
\operatorname{Ext} B\coloneqq \mathcal{S} \cup \mathcal{K} \subset \BV(0,T), \text{ for } B:= \big\{ u \in \BV(0,T) \;\vert\; \TGV(u) \leq 1 \big\}.
\end{align*}
Note that for the unit ball~$\mathcal{B}$ in~$\BV(0,T)/ \mathcal{L}$, see~\eqref{eq:unitball}, and its set of extremal points, Theorem~\ref{thm:tgvextremals}, we have
\begin{align*}
\mathcal{B}=\left\{\,u + \mathcal{L}\;|\;u \in B\,\right\},~\operatorname{Ext} \mathcal{B}=\left\{\,u + \mathcal{L}\;|\;u \in \operatorname{Ext} B\,\right\}.
\end{align*}
Now, given~$\varphi \in L^2(0,T)$, consider the infimizing problem
\begin{align} \label{def:linearprobgen}
\inf_{v \in B} (\varphi,v)_{L^2}.
\end{align}
Problems of this form will play a vital role in the following sections. We point out that the set~$B$ is unbounded in~$\BV(0,T)$, i.e.,~\eqref{def:linearprobgen} does not necessarily admit a minimizer. In the following lemma we show that it is well-posed if~$\varphi$ is orthogonal to the kernel of the~$\TGV$-functional. Moreover, in order to find a solution, it suffices to minimize over the smaller set~$\operatorname{Ext}B$.
\begin{theorem}\label{thm:existlinoverext}
Assume that~$\varphi \in L^2(0,T)$ satisfies~$(\varphi,\ell)_{L^2}=0$ for all~$\ell \in\mathcal{L}$. Then~\eqref{def:linearprobgen} admits a minimizer in~$B$ and there holds
\begin{align*}
\min_{v \in B} (\varphi,v)_{L^2}=\min_{v \in \ext B} (\varphi,v)_{L^2}.\end{align*}
\end{theorem}
\begin{proof}
Let~$v \in B $ be arbitrary but fixed. Using that $T < +\infty$ and the Sobolev embedding \cite[Thm.~3.47]{AmbFusPal00} for $\BV(0,T)$ there holds 
\begin{align*}
(\varphi,v)_{L^2}=(\varphi,v+\ell)_{L^2} \leq \|\varphi\|_{L^2} \|v+\ell\|_{L^2}\leq \|\varphi\|_{L^2} \|v+\ell\|_{L^\infty}\leq c\,\|v+\ell\|_{\BV} \quad \text{for all}~ \ell \in\mathcal{L}
\end{align*}
for some~$c>0$ independent of~$v$ and~$\ell$. In particular, this implies\begin{align*}
(\varphi,v)_{L^2} \leq c \inf_{\ell \in\mathcal{L}}\|v+\ell\|_{\BV},
\end{align*}
i.e.,~$\varphi$ induces a linear and continuous functional~$\Phi \colon \BV(0,T)/\Lc \to \mathbb{R} $ with~
\begin{align*}
\Phi(u+\mathcal{L})=(\varphi,u)_{L^2} \quad \text{for all}~u+\mathcal{L} \in \BV(0,T)/\Lc.
\end{align*}
Now by Proposition \ref{prop:quotientstuff} and the Banach-Alaoglu theorem, the set~$\mathcal{B}$ is weak* compact in~$\BV(0,T)/\Lc$. Hence there exists~$\bar{v} \in B$ with
\begin{align} \label{eq:minforlinquotient}
\Phi(\bar{v}+\mathcal{L})= \min_{v+\mathcal{L}\in \mathcal{B}} \Phi(v+\mathcal{L}).
\end{align}
Now we claim that
\begin{align} \label{eq:linearmininext}
\min_{v+\mathcal{L}\in \mathcal{B}} \Phi(v+\mathcal{L})=\min_{v+\mathcal{L}\in \operatorname{Ext}(\mathcal{B})} \Phi(v+\mathcal{L}).
\end{align} 
In order to prove this, consider the convex and compact image set
\begin{align*}
\Phi(\mathcal{B}) \coloneqq \left\{\,\Phi(v+\mathcal{L})\;|\;v \in B\,\right\} \subset \R,
\end{align*}
for which it is readily verified that
\begin{align*}
\operatorname{Ext} \Phi(\mathcal{B})= \left\{\min_{v+\mathcal{L}\in \mathcal{B}} \Phi(v+\mathcal{L}), \max_{v + \mathcal{L}\in \mathcal{B}} \Phi(v+\mathcal{L})\right\}.
\end{align*}
By invoking~\cite[Lemma 3.2]{BreCar20} we finally get
\begin{align*}
\min_{v+\mathcal{L}\in \mathcal{B}} \Phi(v+\mathcal{L}) \in \left\{\min_{v+\mathcal{L}\in \ext \mathcal{B}} \Phi(v+\mathcal{L}), \max_{v+\mathcal{L}\in \ext \mathcal{B}} \Phi(v+\mathcal{L})\right\}.
\end{align*}
which proves~\eqref{eq:linearmininext}. Now, let~$\bar{v} \in \ext B$ be such that~\eqref{eq:minforlinquotient} holds. Then we have
\begin{align*}
(\varphi,\bar{v})_{L^2}= \Phi(\bar{v}+\mathcal{L}) \leq \Phi(v+\mathcal{L})=(\varphi,v)_{L^2} \quad \text{for all}~ v \in B.
\end{align*}
Together with~$\ext B \subset B$ the proof is finished.
\end{proof}

\subsection{Necessary first-order optimality conditions} \label{subsec:subdiff}
The aim of this section is to prove first-order necessary optimality conditions for Problem~\eqref{def:minprob}. Later, these will be used to infer structural properties of its minimizers. For this purpose, start by defining the~$L^2$-subdifferential of the~$\TGV$-functional at $\bar{u} \in \BV(0,T)$ as 
\begin{align} \label{def:L2sub}
\partial_{2} \TGV(\bar{u}) \coloneqq \left \{\,\varphi \in L^2(0,T)\;|\;(\varphi,u-\bar{u})_{L^2}+\TGV(\bar{u})\leq \TGV(u) \quad \text{for all}~ u \in \BV(0,T)\;\,\right\}.
\end{align}
Note that~$\partial_{2} \TGV(\bar{u})$ can be empty. A full characterization of the subdifferential is given as follows:
\begin{theorem}
\label{thm:subdiff}
Let~$\bar{u} \in \BV(0,T) $ be given and let~$\bar{w} \in \BV(0,T)$ be such that
\begin{align*}
\TGV(\bar{u})= \alpha \|D\bar u- \bar w\|_{M} + \beta \|D \bar w\|_M.
\end{align*}
Moreover, for~$\bar{\varphi} \in L^2(0,T)$ define
\begin{align} \label{def:pandP}
\bar{p}=- \int^{\boldsymbol{\cdot}}_0 \bar{\varphi}(x)~\mathrm{d}x,~\bar{P}=- \int^{\boldsymbol{\cdot}}_0 \bar{p}(x)~\mathrm{d}x.
\end{align}
Then we have~$\bar{\varphi} \in \partial_2 \TGV(\bar{u})$ if and only if
\begin{itemize}
\item There holds
\begin{equation} \label{eq:boundonpandP}
\bar{p} \in H^1_0(0,T),~\bar{P} \in H^2(0,T) \cap H^1_0(0,T),~\|\bar{p}\|_C \leq \alpha,~\|\bar{P}\|_C \leq \beta.
\end{equation}
\item We have
\begin{equation} \label{eq:suppcondition}
\begin{split}
\operatorname{supp} (D\bar{u}-\bar{w})_\pm &\subset \left\{\,x \in (0,T)\;|\;\bar{p}(x)= \pm \alpha\,\right\}, 
\\
\operatorname{supp} (D\bar{w})_\pm &\subset \left\{\,x \in (0,T)\;|\;\bar{P}(x)= \pm \beta\,\right\} .
\end{split}
\end{equation}
\end{itemize}
\end{theorem}
We break down the proof into smaller steps. The following proposition serves as a starting point.
\begin{prop} \label{prop:subdiffpre}
Let~$\bar{u} \in \BV(0,T)$ and~$\bar{\varphi} \in L^2(0,T)$. Then there holds~$\bar{\varphi} \in \partial_2 \TGV(\bar{u})$ if and only if
\begin{align*}
(\bar{\varphi}, \ell)_{L^2}=0,~(\bar{\varphi},\bar{u})_{L^2}=\TGV(\bar{u}),~(\bar{\varphi},v)_{L^2} \leq \TGV(v)
\end{align*} 
for all~$\ell \in \mathcal{L}$,~$v \in \BV(0,T)$.
\end{prop}
\begin{proof}
By definition,~there holds~$\bar{\varphi} \in \partial_2 \TGV(\bar{u})$ if and only if
\begin{align}
(\bar{\varphi},u-\bar{u})_{L^2}+\TGV(\bar{u})\leq \TGV(u) \quad \text{for all}~ u \in \BV(0,T)
\end{align}
Of course, this is equivalent to~$\bar{u}$ being a minimizer of
\begin{align}
\min_{v \in \BV(0,T)} \left \lbrack -(\bar{\varphi},v)_{L^2}+\TGV(v) \right \rbrack.
\end{align}
Since~$\TGV(\cdot)$ is positively one-homogeneous and $\TGV(u+\ell)=\TGV(u)$ for all $u \in \BV(0,T)$ and $\ell \in \Lc$, this holds (with value zero) if and only if
\begin{align}
(\bar{\varphi}, \ell)_{L^2}=0,~(\bar{\varphi},\bar{u})_{L^2}=\TGV(\bar{u}),~(\bar{\varphi},v)_{L^2} \leq \TGV(v)
\end{align} 
for all~$\ell \in \mathcal{L}$ and~$v \in \BV(0,T)$.
\end{proof}
Now we first prove that~$\bar{\varphi}$ is orthogonal to~$\mathcal{L}$ if and only if~$\bar{p}$ and~$\bar{P}$ admit zero boundary conditions.
\begin{lemma} \label{lem:orthogonality}
Let~$\bar{\varphi} \in L^2(0,T)$ be given and let~$\bar{p}$ and~$\bar{P}
$ be defined as in~\eqref{def:pandP}.
The following statements are equivalent:
\begin{itemize}
\item There holds~$(\bar{\varphi},\ell)_{L^2}=0$ for all~$\ell \in \mathcal{L}$.
\item We have~$\bar{p} \in H^1_0(0,T)$ and~$\bar{P} \in H^1_0(0,T) \cap H^2(0,T)$.
\end{itemize}
\end{lemma}
\begin{proof}
By construction we have~$\bar{p} \in H^1(0,T)$,~$\bar{P} \in H^2(0,T)$ and~$\bar{p}(0)=\bar{P}(0)=0$. Consequently it suffices to show that~$(\bar{\varphi},\ell)_{L^2}=0$ for all~$\ell \in \mathcal{L}$ holds if and only if~$\bar{p}(T)=\bar{P}(T)=0$. Testing~$\bar{\varphi}$ with~$u_1= 1_{(0,T)} \in \mathcal{L}$ and~$u_2=x \in \mathcal{L}$ reveals
\begin{align*}
(\bar{\varphi},u_1)_{L^2}=-\bar{p}(T)
\end{align*}
as well as
\begin{align*}
(\bar{\varphi},u_2)_{L^2}= -\bar{p}(T)+ (\bar{p}, u_1)_{L^2}= -\bar{p}(T)-\bar{P}(T)
\end{align*}
by partial integration.
Noting that~$\mathcal{L}=\operatorname{span}\{u_1,u_2\}$ finishes the proof.
\end{proof}
Next, it is shown that the boundedness condition
\begin{align*}
(\bar{\varphi},v)_{L^2} \leq \TGV(v) \quad \text{for all}~ v \in \BV(0,T)\end{align*}
corresponds to certain bounds on the uniform norm of~$\bar{p}$ and~$\bar{P}$, respectively.
\begin{lemma} \label{lem:normboundsonpandP}
Assume that there holds
\begin{align*}
(\bar{\varphi},\ell)_{L^2}=0 \quad \text{for all} ~\ell \in \mathcal{L}.
\end{align*}
Then the following statements are equivalent:
\begin{itemize}
\item[1.] There holds~$(\bar{\varphi},v)_{L^2} \leq \TGV(v)$ for all~$v \in \BV(0,T)$.
\item[2.] We have
\begin{align*}
\|\bar{p}\|_C \leq \alpha,\text{ and }|\bar{P}(x)|\leq \beta \ \text{for all}~ x \in \left\lbrack \frac{\beta}{\alpha}, T- \frac{\beta}{\alpha} \right \rbrack.
\end{align*}
\item[3.] We have~$\|\bar{p}\|_C \leq \alpha$ and~$\|\bar{P}\|_C \leq \beta$.
\end{itemize}
\end{lemma}
\begin{proof}
Note that there holds
\begin{align*}
(\bar{\varphi},v)_{L^2} \leq \TGV(v) \ \text{for all}~v \in \BV(0,T)  \Leftrightarrow \max_{v \in B} (\bar{\varphi},v)_{L^2} \leq 1  \Leftrightarrow \max_{v \in \ext B} (\bar{\varphi},v)_{L^2} \leq 1.
\end{align*}
due to the positive one-homogeneity of the~$\TGV$-functional as well as Theorem~\ref{thm:existlinoverext}.
Testing~$\bar{\varphi}$ with~$\pm S_x$ and~$\pm K_x$ and integrating by parts yields
\begin{align*}
\pm(\bar{\varphi},S_{x})_{L^2}= \pm \bar{p}(x),~(\bar{\varphi},K_{x})_{L^2}=(\bar{p},DK_{x})_{L^2}= \pm \bar{P}(x)
\end{align*}
for every~$x \in (0,T)$. Due to the characterization of~$\ext B$ in Theorem \ref{thm:tgvextremals} we thus conclude the equivalence of~$1.$ and~$2.$. Moreover, we directly get that ~$3.$ implies~$1.$ by the definition of $\TGV$ in \eqref{eq:tgvsupdef} and of $\bar{p}$ and $\bar{P}$ in \eqref{def:pandP}. Hence it only remains to show that~$2.$ implies~$3.$. To this end, we estimate
\begin{align*}
|\bar{P}(x)| \leq \int^x_0 |p_k(x)|~\mathrm{d}x \leq \alpha |x| \leq \beta \quad \text{for all}~ x \in \left(0, \frac{\beta}{\alpha} \right\rbrack 
\end{align*}
as well as
\begin{align*}
|\bar{P}(x)| \leq \int^T_x |p_k(x)|~\mathrm{d}x \leq \alpha |T-x| \leq \beta  \quad \text{for all}~ x \in \left\lbrack T -\frac{\beta}{\alpha},T \right) 
\end{align*}
using
\begin{align*}
\bar{P}(x)= - \int^x_0 \bar{p}(x)~\mathrm{d}x=\int^T_x \bar{p}(x)~\mathrm{d}x-\bar{P}(T)=\int^T_x \bar{p}(x)~\mathrm{d}x.
\end{align*}
Hence we have~$\|\bar{P}\|_C \leq \beta$.
\end{proof}
Finally we require the following result on the support of Radon measures.  
\begin{lemma} \label{lem:extremalitymeasure}
Let~$\mu  \in M(0,T) $ and~$p \in C_0(0,T)$ be given. Moreover let~$\gamma>0$ with~$\|p\|_C \leq \gamma$ be given. Then there holds
\begin{align} \label{eq:extremalcond}
\langle p, \mu \rangle=\gamma\|\mu\|_M
\end{align}
if and only if
\begin{align} \label{eq:suppcondaux}
\operatorname{supp} \mu_\pm \subset \left\{\,x \in (0,T)\;|\;\bar{p}(x)=\pm\gamma\, \right\}  
\end{align}
where~$\mu=\mu_+-\mu_{-}$ denotes the Jordan decomposition.
\end{lemma}
\begin{proof}
Obviously, both,~\eqref{eq:extremalcond} and~\eqref{eq:suppcondaux}, trivially hold if~$\mu=0$. Thus, w.l.o.g, assume that~$\mu\neq 0$. In this case, both conditions can only hold if~$\|p\|_C = \gamma$ due to
\begin{align*}
\langle p, \mu \rangle \leq \|p\|_C \|\mu\|_M \leq \gamma \|\mu\|_M
\end{align*}
and~$\operatorname{supp} \mu \neq \emptyset$.
Now, let us first assume~\eqref{eq:extremalcond}. Since~$\mu \neq 0$ and~$p \neq 0$,~\eqref{eq:suppcondaux} is implied by~\cite[Lemma 3.4]{KKsparse}.
Conversely, assume~\eqref{eq:suppcondaux}. Then we have
\begin{align*}
\langle p, \mu \rangle &= \int^T_0 p(x)~\mathrm{d}\mu(x)= \gamma \int^T_0 ~\mathrm{d}|\mu|(x)=\alpha \|\mu\|_M,
\end{align*}
finishing the proof.
\end{proof}
Putting together the previous observations we finally prove Theorem~\ref{thm:subdiff}.
\begin{proof}[Proof of Theorem \ref{thm:subdiff}]
In virtue of Proposition~\ref{prop:subdiffpre} and Lemmas~\ref{lem:orthogonality} and~\ref{lem:normboundsonpandP} it suffices to show that~\eqref{eq:suppcondition} is equivalent to
\begin{align} \label{eq:extremalTGV}
(\bar{\varphi},\bar{u})_{L^2}=\TGV(\bar{u})
\end{align}
if we have~$\|\bar{p}\|_C \leq \alpha$ and $\|\bar{P}\|_C \leq \beta$. For this purpose note that there holds
\begin{align*}
\TGV(\bar{u})= \alpha \|D\bar{u}-\bar{w}\|_M+\beta \|D\bar{w}\|_M.
\end{align*}
as well as 
\begin{align*}
(\bar{\varphi},\bar{u})_{L^2}= \langle \bar{p}, D\bar{u}-\bar{w} \rangle+\langle \bar{P} , D \bar{w}\rangle
\end{align*}
due to Lemma~\ref{lem:tgvinfconv} and partial integration. Now we point out
\begin{align*}
\langle \bar{p}, D\bar{u}-\bar{w} \rangle \leq \|p\|_C \|D\bar{u}-\bar{w}\|_M,~\langle \bar{P} , D \bar{w}\rangle \leq \|P\|_C \|D\bar{w}\|_M.
\end{align*}
Consequently,~\eqref{eq:extremalTGV} holds if and only if
\begin{align*}
\langle \bar{p}, D\bar{u}-\bar{w} \rangle = \|p\|_C \|D\bar{u}-\bar{w}\|_M,~\langle \bar{P} , D \bar{w}\rangle = \|P\|_C \|D\bar{w}\|_M.
\end{align*}
The equivalence between then~\eqref{eq:suppcondition} and~\eqref{eq:extremalTGV} follows immediately from Lemma~\ref{lem:extremalitymeasure}.
\end{proof}
In the following we silently assume the existence of a minimizer to~\eqref{def:minprob}. The Riesz-representative of the Fr\'echet-derivative of~$f$ at~$u\in L^2(0,T)$ is further denoted by~$\nabla f(u)\in L^2(0,T)$. Using Theorem~\ref{thm:subdiff}, the following first order necessary optimality conditions for Problem~\eqref{def:minprob} are readily derived. 
\begin{theorem}
\label{thm:necessaryopt}
Let~$\bar{u} \in \BV(0,T)$ denote a minimizer to~\eqref{def:minprob} and set
\begin{align} \label{def:pandPopt}
\bar{p}= \int^{\boldsymbol{\cdot}}_0 \nabla f(\bar{u})(x)~\mathrm{d}x, \quad \bar{P}(x)=- \int^{\boldsymbol{\cdot}}_0 \bar{p}(x)~\mathrm{d}x.
\end{align}
Then~$\bar{p}$ and~$\bar{P}$ satisfy~\eqref{eq:boundonpandP} and~\eqref{eq:suppcondition}.
If~$f$ is convex and~\eqref{eq:boundonpandP} as well as~\eqref{eq:suppcondition} are satisfied for $\bar{p}, \bar{P}$ in \eqref{def:pandPopt}, then~$\bar{u}$ is a solution to~\eqref{def:minprob}. 
\end{theorem}
\begin{proof}
Since~$f$ is Fr\'echet-differentiable and~$\bar{u}$ is a minimizer of~\eqref{def:minprob} there holds~$-\nabla f(\bar{u}) \in \partial_2 \TGV(\bar{u})$. According to Theorem~\ref{thm:subdiff}, this is equivalent to~$\bar{p}$ and~$\bar{P}$ satisfying~\eqref{eq:boundonpandP} and~\eqref{eq:suppcondition}. Finally, if~$f$ is convex, then~$-\nabla f(\bar{u}) \in \partial_2 \TGV(\bar{u})$ is also sufficient for optimality. This finishes the proof.
\end{proof}
Similar conditions were already derived in~\cite{PapBre15} for a quadratic~$f$. However their results rely on the Fenchel duality theorem for convex problems. In contrast, we take the direct approach involving the subdifferential in Theorem~\ref{thm:subdiff} which easily extends to the nonconvex case. Exploiting Theorem~\ref{thm:necessaryopt}, we can conclude first structural properties of minimizers to~\eqref{def:minprob}. In particular, every minimizer is affine linear close to the boundary.
\begin{lemma} \label{lem:affineatbound}
Let~$\bar{u} \in \BV(0,T)$ denote a minimizer to~\eqref{def:minprob}. Then there is~$\eps>0$ as well as~$a_1,a_2,b_1,b_2 \in \mathbb{R} $ such that
\begin{align*}
\bar{u}(x)= a_1 x+b_1 \text{ for } x \in (0,\eps),\text{ and }~\bar{u}(x)= a_2 x+b_2 \text{ for }x \in (T-\eps,T).
\end{align*}
\end{lemma}
\begin{proof}
Since~$\bar{p}, \bar{P} \in H^1_0(0,T)$ and~$H^1_0(0,T) \hookrightarrow C_0(0,T)$ there is~$\eps>0$ such that
\begin{align*}
\max\left\{|\bar{p}(x)|/\alpha, |\bar{P}(x)|/\beta \right\} <1 \quad \forall x \in (0,\eps) \cup (T-\eps,T).
\end{align*} 
Hence from~\eqref{eq:suppcondition} we conclude
\begin{align*}
D\bar{w}=0,~D\bar{u}-\bar{w}=0 \quad \text{on}~(0,\eps) \cup (T-\eps,T).
\end{align*}
By integration this implies~$D\bar{u}=\bar{w} \equiv a_1$ and thus~$\bar{u}(x)=a_1c+b_1$ for some~$a_1,a_2 \in \R$ on~$(0,\eps)$. The statement for~$(T-\eps,T)$ follows analogously.
\end{proof}
Moreover, if the regularization parameters are large, then~$\TGV(\bar{u})=0$, i.e.,~$\bar{u}\in \mathcal{L}$. We note that results in this spirit were proved in \cite{PapVal15} for the the denoising case.
\begin{lemma} \label{lem:affineoverall}
Let~$f \colon L^2(0,T) \to \R$ be norm-coercive, i.e., there holds
\begin{align*}
\|u_k\|_{L^2}\rightarrow +\infty \Rightarrow f(u_k)\rightarrow +\infty.
\end{align*}
Moreover assume that~$\nabla f$ maps bounded sets to bounded sets. Then there are~$\alpha^*,\beta^* >0$ such that for all~$\alpha > \alpha^*,\beta > \beta^*$, every solution~$\bar{u}$ to~\eqref{def:minprob} is of the form~$\bar{u}(x)=a_{\bar{u}}x+b_{\bar{u}}$ for some~$a_{\bar{u}},b_{\bar{u}} \in \R$.
\end{lemma}
\begin{proof}
First, we claim that the solution set of~\eqref{def:minprob} is bounded in~$L^2(0,T)$ independently of~$\alpha$ and~$\beta$. If this is not the case, then there are sequences~$\{(\alpha_k,\beta_k)\}_k$ as well as ~$\{\bar{u}_k\}_k$ such that $\bar{u}_k$ is a minimizer of~\eqref{def:minprob} with~$\alpha=\alpha_k,~\beta=\beta_k$ as well as~$\|\bar{u}_k\|_{L^2(\Omega)}\rightarrow+\infty$. Then there also holds
\begin{align*}
+\infty= \limsup_{k \rightarrow \infty} f(\bar{u}_k)\leq \limsup_{k \rightarrow \infty} J(\bar{u}_k)\leq J(0)=f(0)< +\infty
\end{align*}
which yields a contradiction. Now, let~$\alpha,\beta >0$ be arbitrary and let~$\bar{u}$ denote any minimizer of~\eqref{def:minprob} with dual variables~$\bar{p},~\bar{P}$ as defined in~\eqref{def:pandPopt} and~$\bar{w} \in \BV(0,T)$ with
\begin{align*}
\TGV(\bar{u})= \alpha \|D\bar u- \bar w\|_{M} + \beta \|D \bar w\|_M.
\end{align*}
Since~$\nabla f$ maps bounded sets to bounded sets, there holds~$\|\nabla f(\bar{u})\|_{L^2} \leq M$ for some~$M>0$ independent of~$\bar{u}$ as well as of~$\alpha,\beta$. For every~$x \in (0,T)$ we then estimate 
\begin{align*}
|\bar{p}(x)| \leq \int^T_0 |\nabla f(\bar{u})(x)|~\mathrm{d}x \leq T \|\nabla f(\bar{u})\|_{L^2} \leq T M 
\end{align*}
as well as
\begin{align*}
|\bar{P}(x)| \leq T \max_{x \in (0,T)} |\bar{p}(x)| \leq T^2 \|\nabla f(\bar{u})\|_{L^2} \leq T^2 M.
\end{align*}
Hence, for every~$\alpha> TM$ and $\beta > T^2 M$, we conclude
\begin{align*}
D\bar{w}=0,~D\bar{u}-\bar{w}=0 \quad \text{on}~(0,T)
\end{align*}
from~\eqref{eq:suppcondition}. As in Lemma~\ref{lem:affineatbound}, we finally get that~$\bar{u}$ is affine linear finishing the proof for~$\alpha^*=TM$ and $\beta^*=T^2 M$.
\end{proof}
\subsection{Existence of sparse minimizers} \label{subsec:convexrepre}
In this section we put the focus on investigating sufficient conditions for the existence of a sparse minimizer~$\bar{u}$ in the sense of~\eqref{def:sparseuintro}. Two characteristic settings are considered. First, we show that any minimizer of~\eqref{def:minprob} is sparse if the associated dual variables~$\bar{p}$ and~$\bar{P}$ only admite a finite number of global maximizers and minimizers.
\begin{theorem}
\label{thm:finitextremasets}
Let~$\bar{u} \in \BV(0,T)$ denote a minimizer to~\eqref{def:minprob} and let~$\bar{w} \in \BV(0,T)$ be such that
\begin{align*}
\TGV(\bar{u})= \alpha \|D\bar u- \bar w\|_{M} + \beta \|D \bar w\|_M.
\end{align*}
Moreover let~$\bar{p}$ and~$\bar{P}$ be defined as in~\eqref{def:pandPopt} and assume that there are sets $\{\bar{x}^S_j\}^{N_S}_{j=1}$,~$\{\bar{x}^K_j\}^{N_K}_{j=1}$ in $(0,T)$,~$N_S, N_K \geq 0$, with
\begin{align} \label{def:finitemax}
\left\{\,x \in (0,T)\;|\;\bar{p}(x)= \pm \alpha\,\right\}=\left\{\bar{x}^S_j\right\}^{N_S}_{j=1},\quad \left\{\,x \in (0,T)\;|\;\bar{P}(x)= \pm \beta\,\right\}=\left\{\bar{x}^K_j\right\}^{N_K}_{j=1}.
\end{align}
Then~$\bar{u}$ is of the form~\eqref{def:sparseuintro} where
\begin{align*}
\bar{\sigma}^S_j= \bar{p}(\bar{x}^S_j)/\alpha,~j=1, \dots,N_S,\quad \bar{\sigma}^K_j= \bar{P}(\bar{x}^K_j)/\beta,~j=1, \dots,N_K
\end{align*}
and~$(\bar{\lambda}^S,\bar{\lambda}^K,\bar{a},\bar{b})$ is a solution to
\begin{align} \label{eq:subprobopti}
\min_{\substack{{\lambda}^S \in \R^{N_S}_+,\ {\lambda}^K \in \R^{N_K}_+ \\ a,b \in \R}} \left \lbrack f\big(u(\lambda^S,\lambda^K,a,b)\big)+ \sum^{N_S}_{j=1} {\lambda}^S_j+ \sum^{N_K}_{j=1} {\lambda}^K_j \right \rbrack
\end{align}
subject to
\begin{align} \label{def:paramu}
u(\lambda^S,\lambda^K,a,b)= (1/\alpha)\sum^{N_S}_{j=1} \bar{\sigma}^S_j {\lambda}^S_j S_{\bar{x}^S_j}+ (1/\beta)\sum^{N_K}_{j=1} \bar{\sigma}^K_j {\lambda}^K_j K_{\bar{x}^K_j}+ {a}x+{b}.
\end{align}
Moreover there holds
\begin{align*}
\bar{w}= (1/\alpha)\sum^{N_K}_{j=1} \bar{\sigma}^K_j \bar{\lambda}^K_j S_{\bar{x}^K_j}+ \bar{a}, \text{ and }\TGV(\bar{u})= \sum^{N_S}_{j=1} \bar{\lambda}^S_j+ \sum^{N_K}_{j=1} \bar{\lambda}^K_j.
\end{align*} 
\end{theorem}
\begin{proof}
In virtue of Theorem~\ref{thm:necessaryopt} and~\eqref{eq:suppcondition} we get
\begin{align*}
D\bar{w}=(1/\beta) \sum^{N_K}_{j=1} \bar{\sigma}^K_j \bar{\lambda}^K_j \delta_{\bar{x}^K_j}
\end{align*}
for some~$\bar{\lambda}^K \in \R^{N_K}_+$. Hence, by integration, there holds
\begin{align*}
\bar{w}= (1/\beta)\sum^{N_K}_{j=1} \bar{\sigma}^K_j \bar{\lambda}^K_j S_{\bar{x}^K_j}+\bar{a}
\end{align*}
where~$\bar{a} \in \R$.
Then, again invoking~\eqref{eq:suppcondition}, we conclude
\begin{align*}
D\bar{u}=(1/\alpha)\sum^{N_S}_{j=1} \bar{\sigma}^S_j \bar{\lambda}^S_j \delta_{\bar{x}^K_j}+\bar{w}
\end{align*}
and thus integrating yields
\begin{align*}
\bar{u}= (1/\alpha)\sum^{N_S}_{j=1} \bar{\sigma}^S_j \bar{\lambda}^S_j S_{\bar{x}^S_j}+ (1/\beta)\sum^{N_K}_{i=j} \bar{\sigma}^K_j \bar{\lambda}^K_j K_{\bar{x}^K_j}+ \bar{a}x+\bar{b}
\end{align*}
for some~$\bar{\lambda}^S \in \mathbb{R}^{N_S}_+,\bar{b} \in \R$. Next we prove that
\begin{align} \label{eq:equivofTGV}
\TGV(\bar{u})=\sum^{N_S}_{j=1} \bar{\lambda}^S_j+ \sum^{N_K}_{j=1} \bar{\lambda}^K_j.
\end{align}
We readily verify that
\begin{align*}
\TGV\big(u(\lambda^S,\lambda^K,a,b)\big) \leq \sum^{N_S}_{j=1} {\lambda}^S_j+ \sum^{N_K}_{j=1} {\lambda}^K_j
\end{align*}
due to the convexity and positive one-homogeneity of the~$\TGV$-functional as well as~$\TGV(\pm S_x) \leq 1,~\TGV(\pm K_x) \leq 1$ for every~$x \in (0,T)$. Moreover there holds
\begin{align*}
\sum^{N_S}_{j=1} \bar{\lambda}^S_j+ \sum^{N_K}_{j=1} \bar{\lambda}^K_j=-(\nabla f(\bar{u}),\bar{u})_{L^2} \leq \TGV(\bar{u})
\end{align*}
due to
\begin{align*}
\pm \bar{p}(x)=\mp(\nabla f(\bar{u}), S_x)_{L^2},\quad \pm \bar{P}(x)=\mp(\nabla f(\bar{u}), K_x)_{L^2}
\end{align*}
for every~$x\in(0,T)$ as well as Proposition~\ref{prop:subdiffpre}. This proves~\eqref{eq:equivofTGV}.
Finally, combining the previous observations, we arrive at
\begin{align*}
f(\bar{u})+ \sum^{N_S}_{j=1} \bar{\lambda}^S_j+ \sum^{N_K}_{j=1} \bar{\lambda}^K_j= J(\bar{u}) &\leq J\big(u(\lambda^S,\lambda^K,a,b)\big) \\ &\leq f\big(u(\lambda^S,\lambda^K,a,b)\big)+ \sum^{N_S}_{j=1} {\lambda}^S_j+ \sum^{N_K}_{j=1} {\lambda}^K_j
\end{align*}
As a consequence,~$(\bar{\lambda}^S,\bar{\lambda}^K,\bar{a},\bar{b})$ is a minimizer of~\eqref{eq:subprobopti}.
\end{proof}
In particular, this result is applicable if the dual variables are analytic.
\begin{coroll} \label{coroll:analytic}
Let~$\bar{u} \in \BV(0,T)$ be a minimizer to~\eqref{def:minprob}. Moreover assume that the dual variables~$\bar{p},\bar{P}$, see~\eqref{def:pandPopt}, are analytic. Then~$\bar{u}$ is of the form~\eqref{def:sparseuintro}.
\end{coroll}
\begin{proof}
Since the dual variables are nonzero and~$\bar{p}, \bar{P} \in C_0(0,T)$, we conclude that the functions~$\tilde{p}= \bar{p}-\alpha$ and~$\tilde{P}= \bar{P}-\beta$ are analytic and non-constant. Thus, the zero sets of~$\tilde{p}$ and~$\tilde{P}$  cannot have any accumulation points, so they are finite. As a consequence, Theorem~\ref{thm:finitextremasets} is applicable and the statement follows.  
\end{proof}
For the second setting, assume that the smooth part~$f$ of the objective functional is of composite form, i.e.,~$f(u)=F(Hu)$ where~$F$ is a convex function and~$H$ is a linear and continuous operator with finite range. In this setting, the following theorem states the existence of a sparse minimizer~\textit{without} further assumptions on the dual variables.
\begin{theorem}
\label{thm:convexrepresenter}
Let~$f(u)=F(Hu)$ be given where~$H \colon L^2(0,T) \to Y$ is a linear and continuous operator,~$Y\simeq \R^N$ is a Hilbert space , and~$F \colon Y \to \R$ is convex and continuously Fr\'echet-differentiable. Moreover assume that~$F$ is norm-coercive on~$Y$, i.e.,
\begin{align*}
\|y\|_Y \rightarrow +\infty \Rightarrow F(y) \rightarrow +\infty
\end{align*}
and~$H(L^2(0,T))=Y$. Define the subspace
\begin{align*}
H (\mathcal{L}) \coloneqq \left \{\,Hu\;|\;u \in \mathcal{L}\,\right\}.
\end{align*}
Then Problem~\eqref{def:minprob} admits a solution~$\bar{u}$ of the form \eqref{def:sparseuintro} with~$N_k+N_S \leq N-\operatorname{dim}(H (\mathcal{L}))$ and
\begin{align*}
\TGV(\bar{u})=\sum^{N_S}_{j=1} \bar{\lambda}^S_j+ \sum^{N_K}_{j=1} \bar{\lambda}^K_j.
\end{align*}
\end{theorem}
\begin{proof}
Given the stated assumptions, it is readily verified that~\cite[Theorem 1]{BreCar20} can be applied to Problem~\ref{def:minprob}. This yields the existence of at least one minimizer~$\bar{u}$ to~\eqref{def:minprob} with 
\begin{align*}
\bar{u}= \bar{\psi}+ \sum^{\bar{N}}_{j=1} \bar{\lambda}_j \bar{u}_j,~\TGV(\bar{u})=\sum^{\bar{N}}_{j=1}  \bar{\lambda}_j
\end{align*}
for some~$\bar{\psi} \in \mathcal{L}$,~$\bar{\lambda}_j>0$,~$\bar{u}_j \in \BV(0,T)$ with~$\bar{u}_j +\mathcal{L} \in \operatorname{Ext} \mathcal{B}$ and~$\bar{N} \leq \operatorname{dim}( Y / H(\mathcal{L}))$. The claim then follows by the definition of~$\mathcal{L}$, the characterization of~$\operatorname{Ext} \mathcal{B}$ in Theorem~\ref{thm:tgvextremals} and~$\operatorname{dim}( Y / H(\mathcal{L}))=\operatorname{dim}( Y)-\operatorname{dim}( H(\mathcal{L}))$.
\end{proof}
\begin{remark}
It is worth pointing out that the results of Theorem~\ref{thm:convexrepresenter} remain valid for proper, convex and lower semicontinuous functionals~$F \colon Y \to \R$.
\end{remark}
\begin{remark} \label{rem:counterexample}
Of course, Theorem~\ref{thm:finitextremasets} and~\ref{thm:convexrepresenter} raise the question whether there holds 
\begin{align*}
\TGV(u)= \sum^{N_K}_{j=1} \lambda^K_j+\sum^{N_S}_{j=1} \lambda^S_j
\end{align*}
for every~$u\in \BV(0,T)$ of the form~\eqref{def:paramu}. This is~\textit{not} the case as we will see in the following. For this purpose let~$\alpha,\beta>0$,~$T>0$ as well as~$\bar{x}_1,\bar{x}_2 \in [T/2,T-\beta/\alpha)$,~$\bar{x}_2>\bar{x}_1$, be such that
\begin{align*}
|\bar{x}_1-\bar{x}_2| < \frac{\beta}{2\alpha}.
\end{align*}
Moreover, fix coefficients~$\bar{\lambda}_1,\bar{\lambda}_2>0$ with
\begin{align*}
|\bar{\lambda}_1-\bar{\lambda}_2|<\frac{\beta}{2\alpha},~|T-\bar{x}_2| < \bar{\lambda}_1+\bar{\lambda}_2.
\end{align*}
Set~$\bar{u}= \bar{\lambda}_1 K_{\bar{x}_1}/\beta-\bar{\lambda}_2 K_{\bar{x}_2}/\beta$. Then we readily calculate
\begin{align*}
\TGV(\bar{u}) \leq \frac{\alpha}{\beta} \|\bar{\lambda}_1 S_{\bar{x}_1}-\bar{\lambda}_2 S_{\bar{x}_2}\|_{L^1}&=\frac{\alpha}{\beta}\left(\bar{\lambda}_1 |\bar{x}_2-\bar{x}_1|+ |\bar{\lambda}_2-\bar{\lambda}_1||T -\bar{x}_2|\right)
\\& < \bar{\lambda}_1/2+ |T-\bar{x}_2|/2 \\
&< \bar{\lambda}_1+\bar{\lambda}_2
\end{align*}
which gives a counterexample.
\end{remark}
\section{A minimization algorithm} \label{sec:algorithm}
This section is devoted to the analysis of a simple solution algorithm for Problem~\eqref{def:minprob}. In particular, the presented procedure does not rely on a discretization of the underlying functions. This is a challenging problem for a variety of reasons. First, it requires to work directly on the non-reflexive space~$\BV(0,T)$. Second, we also have to take into account the nonsmoothness of the~$\TGV$-functional. Finally, note that evaluating~$\TGV(u)$ already requires the solution of a nonsmooth minimization problem. Our method follows the spirit of \cite{BredWal} and alternates between the update of a set~$\mathcal{A}_k$ in~$\ext B$ via solving a linear minimization problem over~$B$ and improving the iterate~$u_k \in \operatorname{cone}(\mathcal{A}_k)+ \mathcal{L}$. The former can be implemented efficiently using Theorem~\ref{thm:existlinoverext} while the latter is done by minimizing a~$\ell_1$-regularized surrogate functional over the finite dimensional set~$\operatorname{cone}(\mathcal{A}_k)+\mathcal{L}$.  This completely avoids evaluating the~$\TGV$-functional throughout the iterations and eventually guarantees the subsequential convergence of~$\{u_k\}_k$ towards stationary points of~\eqref{def:minprob}.\subsection{The procedure} \label{subsec:proc}
Given a finite ordered set~$ \mathcal{A}=\{u_j\}^N_{j=1}$ in~$\ext B$ consider the following finite dimensional minimization problem:
\begin{align} \label{def:subprob}
\min_{\lambda \in \mathbb{R}^N_+, \ell \in \mathcal{L} } \left \lbrack f\left(\sum^N_{j=1} \lambda_j u_j+ \ell\right)+ \sum^N_{j=1} \lambda_j \right \rbrack. \tag{$\mathcal{P}(\mathcal{A})$}
\end{align}
Note that the objective functional in~\eqref{def:subprob} constitutes an upper bound on~$J$ for every function~$u \in \operatorname{cone}(\mathcal{A})+\mathcal{L}$.
\begin{lemma}
Let~$(\lambda,\ell) \in \R^N_+ \times \mathcal{L}$ be arbitrary and set 
\begin{align*}
u= \sum^N_{j=1} \lambda_j u_j+\ell \in \operatorname{cone}(\mathcal{A})+ \mathcal{L}.
\end{align*}
Then there holds
\begin{align*}
J(u) \leq f(u)+ \sum^N_{j=1} \lambda_j.
\end{align*}
\end{lemma}
\begin{proof}
Since the~$\TGV$-functional is convex and one-homogeneous it is also sublinear. Hence there holds
\begin{align} \label{eq:upperforTGV}
\TGV(u) \leq \sum^N_{j=1} \lambda_j \TGV(u_j)= \sum^N_{j=1} \lambda_j
\end{align}
using~$\TGV(u_j)=1$ as well as componentwise nonnegativity of~$\lambda$. The claim readily follows.
\end{proof}
Now we propose a greedy method for the computation of a stationary point of~\eqref{def:minprob} which works as follows: Let the current~\textit{active set}~$\mathcal{A}_k=\{u^k_j\}^{N_k}_{j=1}$,~$N_k \geq 0$, in~$\ext B$ as well as the iterate~$u_k$ be given. Assume that~$u_k$ is constituted by a minimizing pair~$(\lambda^k, \ell^k)$ to~$(\mathcal{P}(\mathcal{A}_k))$ as
\begin{align} \label{def:constiterate}
u_k= \sum^{N_k}_{j=1} \lambda^k_j u^k_j+ \ell_k.
\end{align}
Then, we first compute the current dual variables defined by
\begin{align} \label{eq:iterduals}
p_k= \int^{\boldsymbol{\cdot}}_0 \nabla f(u_k)(s)~\mathrm{d}s,~P_k= -\int^{\boldsymbol{\cdot}}_0 p_k(s)~\mathrm{d}s \in C_0(0,T)
\end{align}
as well as points~$\hat{x}^S_k, \hat{x}^K_k \in (0,T)$ with
\begin{align*}
|p_k(\hat{x}^S_k)|=\|p_k\|_C,~|P_k(\hat{x}^K_k)|=\|P_k\|_C.
\end{align*} 
If~$\max\{\|p_k\|_C/\alpha, \|P_k\|_C/\beta\} \leq 1$ the method stops with~$\bar{u}=u_k$ being a stationary point for~\eqref{def:minprob}.
\begin{prop} \label{prop:finitestep}
Let~$(\lambda^k, \ell^k)\in \mathbb{R}^N_+ \times \mathcal{L}$ denote a minimizer of~$(\mathcal{P}(\mathcal{A}_k))$ and set
\begin{align*}
u_k=\sum^{N_k}_{j=1} {\lambda}^k_j u^k_j+ \ell^k.
\end{align*}
Define the dual variables~$p_k$ and~$P_k$ as in~\eqref{eq:iterduals}. Then these satisfy $p_k \in H^1_0(0,T)$ and~$P_k \in H^2(0,T) \cap H^1_0(0,T)$, and if~$\max\{\|p_k\|_C/\alpha,\|P_k\|_C/\beta\}\leq 1$ then~$u_k$ is a stationary point of~\eqref{def:minprob} and~$\TGV(u_k)= \sum^{N_k}_{j=1} \lambda^k_j$.
\end{prop}
\begin{proof}
The necessary first-order optimality conditions for Problem~$(\mathcal{P}(\mathcal{A}_k))$ are given by
\begin{align} \label{eq:firstordersubprob1}
(\nabla f(u_k),\ell)_{L^2}= 0\  \text{for all}~\ell \in \mathcal{L},\quad-(\nabla f(u_k),u^k_j)_{L^2} \leq 1 \ \text{for all}~j=1,\dots,N_k\end{align}
implying~$p_k \in H^1_0(0,T)$ and~$P_k \in H^2(0,T) \cap H^1_0(0,T)$ by Lemma~\ref{lem:orthogonality}, as well as
\begin{align} \label{eq:firstordersubprob2}
-(\nabla f(u_k),u_k)_{L^2}=\sum^{N_k}_{j=1} \lambda^k_j.
\end{align}
Moreover, if
\begin{align*}
\max\{\|p_k\|_C/\alpha,\|P_k\|_C/\beta\}\leq 1,
\end{align*}
then we have 
\begin{align*}
-(\nabla f(u_k),u) \leq \TGV(u) \quad \text{for all}~u \in \BV(0,T)
\end{align*}
according to Lemma~\ref{lem:normboundsonpandP}. Combining these observations with~\eqref{eq:upperforTGV} we conclude
\begin{align*}
\sum^{N_k}_{j=1} \lambda^k_j=-(\nabla f(u_k),u_k)_{L^2} \leq \TGV(u_k) \leq \sum^{N_k}_{j=1} \lambda^k_j
\end{align*}
Thus,~in virtue of Proposition~\ref{prop:subdiffpre},~$u_k$ is a stationary point of~\eqref{def:minprob} and~$\TGV(u_k)=\sum^{N_k}_{j=1}\lambda^k_j$.
\end{proof}
Otherwise, if~$\max\{\|\bar{p}\|_C/\alpha,\|\bar{P}\|_C/\beta\}> 1$, define
\begin{align} \label{def:newgamma}
\widehat{u}_k= \begin{cases} 
(\sign(p_k(\hat{x}^S_k))/\alpha)\, S_{\hat{x}^S_k} & \|p_k\|_C/\alpha \geq \|P_k\|_C/\beta \\
-(\sign(P_k(\hat{x}^K_k)/\beta)\, L_{\hat{x}^K_k} & \text{else}.
\end{cases}
\end{align}
and update the active set by adding~$\widehat{u}_k$ to it. More in detail, set
\begin{align} \label{eq:intermediateactive}
\mathcal{A}^+_k= \mathcal{A}_k \cup \{ \widehat{u}_k\}=\{u^{k,+}_j\}^{N^+_k}_{j=1}
\end{align}
where we denote
\begin{align*}
N^+_k=N_k+1,~u^{k,+}_j=u^k_j~\text{for all}~j=1,\dots,N,~u^{k,+}_{N^+_k}=\widehat{u}_k
\end{align*}
for abbreviation.
The following proposition serves several purposes. First, it shows that this update step is well-posed, i.e.,~$\widehat{u}_k$ is indeed an element of~$\operatorname{Ext}B$. Second, it relates the choice of~$\widehat{u}_k$ to the minimization of a particular linear functional over~$B$.
\begin{prop} \label{prop:maxcertifies}
Let~$\widehat{u}_k$ be defined as in~\eqref{def:newgamma}. Then there holds~$\widehat{u}_k \in \operatorname{Ext}B$ as well as
\begin{align} \label{eq:linearmin}
-\max{\{\|p_k\|_C/\alpha, \|P_k\|_C/\beta\}}=( \nabla f(u_k),\widehat{u}_k)_{L^2}=\min_{v \in B}(\nabla f(u_k) ,v)_{L^2}.
\end{align}
Moreover, there exists~$C_{\widehat{u}}>0$ with~$\|\widehat{u}_k\|_{L^2} \leq C_{\widehat{u}}$ for all~$k \geq 1$.
\end{prop}
\begin{proof}
First, note that~$\pm S_x /\alpha \in \operatorname{Ext} B  $ for all~$x \in (0,T)$. Thus, in order to prove that~$\widehat{u}_k \in \operatorname{Ext}B $, it suffices to show~$\hat{x}^K_k \in (\beta/\alpha, T-\beta/\alpha)$ if~$\|p_k\|_C/\alpha < \|P_k\|_C/\beta$. In this case we readily verify
\begin{align*}
|P_k(x)| \leq  \int^x_0 |p_k(x)|~\mathrm{d} x \leq |x|  \|p_k\|_C < \frac{\alpha |x|}{\beta} \|P_k\|_C
\end{align*}
as well as
\begin{align*}
|P_k(x)| \leq  \int^T_x |p_k(x)|~\mathrm{d} x \leq |T-x| \|p_k\|_C < \frac{\alpha |T-x|}{\beta} \|P_k\|_C
\end{align*}
by noting that
\begin{align*}
P_k(x)= -\int^x_0 p_k(x)~\mathrm{d}x,\quad 0= -\int^T_0 p_k(x)~\mathrm{d}x=P_k(x)-\int^T_{x} p_k(x)~\mathrm{d}x 
\end{align*}
for every~$x \in (0,T)$. Here we made explicit use of~$P_k \in C_0(0,T)$ and~$\|p_k\|_C/\alpha < \|P_k\|_C/\beta$. Consequently we get~$|P_k(x)|< \|P_k\|_C$ for all~$x \in [0,T] \setminus (\beta/\alpha, T-\beta/\alpha) $ and thus ~$\hat{x}^K_k \in (\beta/\alpha, T-\beta/\alpha)$.

Next we prove~\eqref{eq:linearmin}. To start we invoke~\eqref{eq:firstordersubprob1} and Theorem~\ref{thm:existlinoverext} to arrive at
\begin{align*}
\min_{v \in B}(\nabla f(u_k) ,v)_{L^2}&=\min_{v \in \operatorname{Ext}(B)}(\nabla f(u_k) ,v)_{L^2}.
\end{align*}
Moreover, from the characterization of~$\operatorname{Ext}(B)$ in Theorem \ref{thm:tgvextremals}, we have
\begin{align*}
\min_{v \in \operatorname{Ext}(B)}&(\nabla f(u_k) ,v)_{L^2}\\ &= \min \left\{\min_{  \substack{x\in (0,T),\\ \sigma \in \{-1,1\}}} \sigma (\nabla f(u_k) ,S_x)_{L^2}/\alpha,\ \inf_{ \substack{ x\in (\beta/\alpha,T-\beta/\alpha),\\ \sigma \in \{-1,1\}}} \sigma (\nabla f(u_k),K_x)_{L^2}/\beta\right\}.
\end{align*}
By partial integration we now verify
\begin{align*}
\min_{  \substack{x\in (0,T),\\ \sigma \in \{-1,1\}}} \sigma (\nabla f(u_k) ,S_x)_{L^2}/\alpha=\min_{  \substack{x\in (0,T),\\ \sigma \in \{-1,1\}}} -\frac{\sigma p_k(x)}{\alpha} \geq - \frac{\|p_k\|_C}{\alpha}
\end{align*}
as well as
\begin{align*}
\inf_{ \substack{ x\in (\beta/\alpha,T-\beta/\alpha),\\ \sigma \in \{-1,1\}}} \sigma (\nabla f(u_k) ,K_x)_{L^2}/\beta= \inf_{ \substack{ x\in (\beta/\alpha,T-\beta/\alpha),\\ \sigma \in \{-1,1\}}} \frac{\sigma P_k(x)}{\beta} \geq  -\frac{\|P_k\|_C}{\beta}.
\end{align*}
In summary, these observations yield
\begin{align*}
\min_{v \in B}(\nabla f(u_k) ,v)_{L^2} \geq -\max\left\{\|p_k\|_C/\alpha, \|P_k\|_C/\beta \right\}.
\end{align*}
The proof is now finished noting that
\begin{align*}
(\nabla f(u_k) ,\widehat{u}_k)_{L^2}&= \begin{cases} 
-\sign(p_k(\hat{x}^s_k)) \frac{p_k(\hat{x}^s_k)}{ \alpha} & \|p_k\|_C/\alpha \geq \|P_k\|_C/\beta \\
-\sign(P_k(\hat{x}^\ell_k)) \frac{P_k(\hat{x}^\ell_k)}{ \beta} & \|p_k\|_C/\alpha < \|P_k\|_C/\beta \\
\end{cases} \\
&= \begin{cases} 
-\|p_k\|_C & \|p_k\|_C/\alpha \geq \|P_k\|_C/\beta \\
-\|P_k\|_C/\beta & \|p_k\|_C/\alpha < \|P_k\|_C/\beta \\
\end{cases}
\\ &= -\max\left\{\|p_k\|_C/\alpha, \|P_k\|_C/\beta \right\}.
\end{align*}
Finally,~$\|\widehat{u}_k\|_{L^2} \leq C_{\widehat{u}}$ for some~$C_{\widehat{u}}>0$ and all~$k \geq 1$ follows by a direct computation.
\end{proof}
Subsequently, the next iterate is found as
\begin{align} \label{eq:updateit}
u_{k+1}= \sum^{N^+_k}_{j=1} \lambda^{k,+} u^{k,+}_j+ \ell^{k,+}
\end{align}  
for a minimizer~$(\lambda^{k,+}, \ell^{k,+})$ of~$(\mathcal{P}(\mathcal{A}^+_k))$.
Now, the active set is pruned by removing all extreme points which were assigned a zero weight, i.e.,
\begin{align} \label{eq:updateofactive}
\mathcal{A}_{k+1}= \mathcal{A}^+_k \setminus \left\{\,u^{k,+}_j\;\middle|\;\lambda^{k,+}_j=0,~j \in \{1,\dots, N^+_k\}\,\right\}
\end{align}
and the iteration counter~$k$ is incremented by one.
Observe that this ensures the well-posedness of Algorithm~\ref{alg:accgcg}, i.e.,~$u_{k+1}$ is constituted by a minimizing pair~$(\lambda^{k+1},\ell^{k+1})$ in the sense of~\eqref{def:constiterate}.
\begin{prop}
Let~$u_{k+1}$ and~$\mathcal{A}_{k+1}$ be defined as in~\eqref{eq:updateit} and~\eqref{eq:updateofactive}, respectively. Denoting~$\mathcal{A}_{k+1}=\{u^{k+1}_j\}^{N_{k+1}}_{j=1}$, there is a minimizing pair~$(\lambda^{k+1},\ell^{k+1})$ to~$(\mathcal{P}(\mathcal{A}_{k+1}))$ such that~$\lambda^{k+1}_j>0$,~$j=1, \dots,N_{k+1}$ as well as
\begin{align*}
u_{k+1}= \sum^{N_{k+1}}_{j=1} \lambda^{k+1}_j u^{k+1}_j+\ell^{k+1}, \quad \sum^{N^+_k}_{j=1} \lambda^{k,+}_j=\sum^{N_{k+1}}_{j=1} \lambda^{k+1}_j,\quad \min (\mathcal{P}(\mathcal{A}^+_k))=\min (\mathcal{P}(\mathcal{A}_{k+1})) .
\end{align*}
\end{prop}
\begin{proof}
Introduce the set of indices
\begin{align*}
\mathcal{I}_k \coloneqq \left\{\,j \in \{1,\dots, N^+_k\}\;\middle|\;\lambda^{k,+}_j >0\right\}.
\end{align*}
By construction there holds
\begin{align*}
\mathcal{A}_{k+1}= \left\{\,u^{k,+}_j\;\middle|\;j \in \mathcal{I}_k\,\right\} \subset \mathcal{A}^+_k
\end{align*}
as well as
\begin{align*}
u_{k+1}= \sum_{j \in \mathcal{I}_k} \lambda^{k,+}_j u^{k,+}_j+ \ell^{k,+} \in \operatorname{cone}(\mathcal{A}_{k+1})+\mathcal{L}
\end{align*}
and
\begin{align*}
\min (\mathcal{P}(\mathcal{A}_{k+1}))\leq f(u_{k+1})+ \sum_{j \in \mathcal{I}_{k}} \lambda^{k,+}_j=f(u_{k+1})+ \sum^{N^+_k}_{j=1} \lambda^{k,+}_j= \min (\mathcal{P}(\mathcal{A}^+_k)) \leq \min (\mathcal{P}(\mathcal{A}_{k+1})).
\end{align*}
This finishes the proof.
\end{proof}

Finally, the iteration counter is incremented by one, i.e.,~$k\coloneqq k+1$ and the next iteration starts. The individual steps of the method are again summarized in Algorithm~\ref{alg:accgcg}. Note that we have incorporated a "zeroth" iteration without an extreme point insertion step to ensure that~$p_1,P_1 \in C_0(0,T)$.
\begin{algorithm}
\begin{flushleft}
\hspace*{\algorithmicindent} \textbf{Input:}~$\mathcal{A}_0=\{u^0_j\}^{N_0}_{j=1}$ in $\operatorname{Ext}B$.
\\
\hspace*{\algorithmicindent} \textbf{Output:} Stationary point~$\bar{u}$ to~\eqref{def:minprob}.
\end{flushleft}
\begin{algorithmic}
\STATE 1. Compute a minimizer~$(\lambda^1, \ell^1)$ to~$(\mathcal{P}(\mathcal{A}_0))$ and set
\begin{align*}
u_1=\sum^{N_0}_{j=1} \lambda^1_j u^0_j+\ell^1,~\mathcal{A}_{1}= \mathcal{A}_0 \setminus \left\{\,u^0_i\;|\;\lambda^1_j=0,~j \in \{1,\dots, N_0\}\,\right\}
\end{align*}
\FOR {$k=1,2,\dots$}
\vspace*{0.3em}
\STATE 2. Compute~$p_k, P_k \in C_0(0,T)$ as in~\eqref{eq:iterduals} and~$\hat{x}^S_k,~\hat{x}^K_k \in (0,T) $ with
\begin{align*}
|p_k(\hat{x}^S_k)|=  \|p_k\|_C,~|P_k(\hat{x}^K_k)|= \|P_k\|_C
\end{align*}
\IF{$\|p_k\|_C\leq \alpha,~\|P_k\|_C\leq \beta$}
\vspace*{0.3em}
\STATE 3. Terminate with~$ \bar{u}=u_k$ a stationary point to~\eqref{def:minprob}.
\vspace*{0.3em}
\ELSE
\vspace*{0.3em}
\STATE 4. Define~$\widehat{u}_k$ and $\mathcal{A}^+_k=\{u^{k,+}_j\}^{N^+_k}_{j=1}$ according to~\eqref{def:newgamma} and~\eqref{eq:intermediateactive}.

\vspace*{0.3em}
\STATE 5. Compute a minimizer~$(\lambda^{k,+}, \ell^{k,+})$ to~$(\mathcal{P}(\mathcal{A}^+_k))$, set
\begin{align*}
u_{k+1}=\sum^{N^+_k}_{j=1} \lambda^{k,+}_j u^{k,+}_j+\ell^{k,+},~\mathcal{A}_{k+1}= \mathcal{A}^+_k \setminus \left\{\,u^{k,+}_j\;\middle|\;\lambda^{k,+}_j=0,~j \in \{1,\dots, N^+_k\}\,\right\}
\end{align*}
and increment~$k=k+1$.
\ENDIF
\ENDFOR
\end{algorithmic}
\caption{Solution algorithm for~\eqref{def:minprob}}
\label{alg:accgcg}
\end{algorithm}
\begin{remark} \label{rem:altsubrpob}
It is worth pointing out that Algorithm~\ref{alg:accgcg} does not guarantee the monotonicity of the objective functional values~$\{J(u_k)\}_k$. More in detail, by construction, we only get~$J(u_k)\leq \min (\mathcal{P}(\mathcal{A}_k))$ as well as~$\min (\mathcal{P}(\mathcal{A}_{k+1}))\leq  \min (\mathcal{P}(\mathcal{A}_{k}))$ for every~$k \geq 1$. This observation could serve as a motivation to consider the alternative subproblem
\begin{align} \label{eq:altsubprob}
\min_{\lambda \in \mathbb{R}^N_+, \ell \in \mathcal{L} } \left \lbrack J \left( \sum^{N}_{j=1} \lambda_j u_j+ \ell \right) \right \rbrack. \tag{$\mathcal{P}_{\text{TGV}}(\mathcal{A})$}
\end{align}
instead of~\eqref{def:subprob}. In general, this leads to a different method since~$\min(\mathcal{P}_{\text{TGV}}(\mathcal{A}))<\min(\mathcal{P}(\mathcal{A}))$ in some cases. For a particular example based on the functions of Remark \ref{rem:counterexample} we refer to Appendix~\ref{sec:counterexample}. While solving~$(\mathcal{P}_{\text{TGV}}(\mathcal{A}^+_k))$ in step 5 ensures 
\begin{align*}
J(u_{k+1})= \min(\mathcal{P}_{\text{TGV}}(\mathcal{A}_{k+1}))\leq\mathcal{P}_{\text{TGV}}(\mathcal{A}_{k})=J(u_k)
\end{align*}
for every~$k\geq 1$, it also requires tailored solution algorithms to handle the appearing~$\TGV$-functional. In contrast,~$(\mathcal{P}(\mathcal{A}^+_k))$ represents a smooth minimization problem with inequality constraints which can be tackled by commonly available black-box solvers.
\end{remark}
\subsection{Convergence rates} \label{subsec:convanalysis}
In the following, we always denote by 
\begin{align*}
\mathcal{A}_k=\{u^k_j\}^{N_k}_{j=1} \subset \ext B,~u_k=\sum^{N_k}_{j=1} \lambda^k_j u^k_j,~(\lambda^k,\ell^k) \in \argmin (\mathcal{P}(\mathcal{A}_k))
\end{align*} 
the active set and iterate generated by Algorithm~\ref{alg:accgcg} in iteration~$k\geq 1$. The associated dual variables~$p_k \in H^1_0(0,T)$,~$P_k \in H^2(0,T) \cap H^1_0(0,T)$ as well as the new candidate function~$\widehat{u}_k$ are defined as in~\eqref{eq:iterduals} and~\eqref{def:newgamma}, respectively. The following assumptions on the objective functional~$j$ are made throughout this section.
\begin{ass} \label{ass:convergence}
Assume that:
\begin{itemize}
\item[\textbf{A1}] The function~$f\colon L^2(0,T) \to \R$ is continuously Fr\'echet differentiable.
\item[\textbf{A2}] The Riesz-representative~$\nabla f\colon L^2(0,T) \to L^2(0,T)$ of the derivative is Lipschitz continuous, i.e., there is~$L_f>0$ with
\begin{align*}
\|\nabla f(u_1)-\nabla f(u_2)\|_{L^2} \leq L_f \|u_1-u_2\|_{L^2} \quad \text{for all}~u_1,u_2 \in L^2(0,T).
\end{align*}
\item[\textbf{A3}] There holds~$f(u)\geq -C_f$ for all~$u \in L^2(0,T)$ and some~$C_f>0$.
\item[\textbf{A4}] The sublevel sets of~$J$ are bounded in~$\BV(0,T)$.
\end{itemize}
\end{ass}
By standard arguments, it is readily verified that these assumptions guarantee the existence of at least one global minimizer in Problem~\eqref{def:minprob}.
\begin{prop} \label{prop:existence}
Let Assumption~\ref{ass:convergence} hold. Then Problem~\eqref{def:minprob} admits at least one minimizer.
\end{prop}
\begin{proof}
Let~$\{u_j\}_j$ denote an infimizing sequence for~\eqref{def:minprob}, i.e.,
\begin{align*}
\lim_{j \rightarrow \infty}J(u_j)=\inf_{u \in\BV(0,T)} J(u) \geq -C_f.
\end{align*}
According to Assumption~\ref{ass:convergence}~(\textbf{A4}),~$\{u_j\}_j$ is bounded in~$BV(0,T)$. Thus, it admits at least one subsequence, denoted by the same index with~$u_j \rightharpoonup^* \bar{u}$ for some~$\bar{u} \in\BV(0,T)$. This implies~$u_j \rightarrow \bar{u}$ in~$L^2(0,T)$ due to~$\BV(0,T) \hookrightarrow_c L^2(0,T)$. Due to the weak lower semicontinuity of the~$\TGV$-functional as well as the continuity of~$f$ we finally arrive at
\begin{align*}
\inf_{u\in\BV(0,T)} J(u)=\liminf_{j \rightarrow \infty} J(u_j) \geq J(\bar{u}) \geq \inf_{u\in\BV(0,T)} J(u).
\end{align*}
This proves that~$\bar{u}$ is a minimizer of~\eqref{def:minprob}.
\end{proof}
Throughout this section we silently assume that Algorithm~\ref{alg:accgcg} does not converge after finitely many iterations but generates a sequence~$\{u_k\}_k$ in~$\BV(0,T)$. In order to measure the non-stationary of the iterate~$u_k$ consider the following~\textit{constraint violation}
\begin{align} \label{def:defpsi}
\Psi(u_k):=M_0 \left( \max\left\{\|p_k\|_C/\alpha,\|P_k\|_C/\beta\right\} -1 \right)~\text{where}~M_0=f(0)+C_f.
\end{align}
Note that we have~$\Psi(u_k)>0$, due to the termination criterion of Algorithm~\ref{alg:accgcg}, as well as~$\TGV(u_k)\leq M_0$ for all~$k \geq 1$ since
\begin{align} \label{eq:TGVuk}
-C_f+ \TGV(u_k) \leq J(u_k) \leq \min (\mathcal{P}(\mathcal{A}_k))\leq f(0). 
\end{align}
The following proposition relates~$\Psi(u_k)$ to the ``per-iteration descent'' of Algorithm~\ref{alg:accgcg} with respect to the residual
\begin{align} \label{def:residual}
\widehat{r}_J(u_k) \coloneqq f(u_k)+\sum^{N_k}_{j=1} \lambda^k_j- \min_{u \in\BV(0,T)} J(u).
\end{align}

\begin{prop} \label{prop:descentlemma}
Let~$\Psi(u_k)$ and~$\widehat{r}_J(u_{k})$ be defined as in~\eqref{def:defpsi} and~\eqref{def:residual}, respectively. There holds
\begin{align} \label{eq:descentforanys}
\widehat{r}_J(u_{k+1}) - \widehat{r}_J(u_{k}) \leq -s \Psi(u_k)+ \frac{s^2L_f}{2} \|u_k-\widehat{u}_{k}\|^2_{L^2} \quad \text{for all}~s \in [0,1]\end{align}
for all~$k\geq1$.
In particular, we have
\begin{align} \label{eq:optdescent}
\widehat{r}_J(u_{k+1}) - \widehat{r}_J(u_{k}) \leq - \max \left\{\frac{\Psi(u_k)^2}{2L_f\|u_k-\widehat{u}_{k}\|^2_{L^2}},\Psi(u_k)-\frac{L_f}{2} \|u_k-\widehat{u}_{k}\|^2_{L^2}\right\}.
\end{align}
and thus~$\min(\mathcal{P}(\mathcal{A}_{k+1}))<\min(\mathcal{P}(\mathcal{A}_{k}))$ for all~$k \geq 1$.
\end{prop}
\begin{proof}
For every~$s\in[0,1]$ define
\begin{align*}
 \tilde{u}_{k,s} \coloneqq sM_0 \widehat{u}_k+(1-s)\sum^{N_k}_{j=1} \lambda^k_j u^k_j+ \ell^k \in \operatorname{cone}(\mathcal{A}^+_k)+\mathcal{L}.
\end{align*}
By construction there holds
\begin{align*}
\widehat{r}_J(u_{k+1})-\widehat{r}_J(u_{k}) \leq f(\tilde{u}_{k,s})-f(u_{k})+s\left(M_0- \sum^{N_k}_{j=1} \lambda^k_j \right) 
\end{align*}
Utilizing a Taylor's expansion as well as the Lipschitz continuity of~$\nabla f$ we get
\begin{align*}
f(\tilde{u}_{k,s})-f(u_{k}) \leq -s (\nabla f(u_k),u_k-M_0 \widehat{u}_k)_{L^2}+ \frac{s^2L_f}{2} \|u_k-\widehat{u}_k\|^2_{L^2}
\end{align*}
as in the proof of the standard descent lemma, see~\cite[(1.2.5)]{nesterov}. Then, using~\eqref{eq:firstordersubprob1} as well as~\eqref{eq:linearmin}, we arrive at
\begin{align*}
-s (\nabla f(u_k),u_k-M_0 \widehat{u}_k)_{L^2}=s \left( \sum^{N_k}_{j=1} \lambda^k_j-M_0 \max \{\|p_k\|_C/\alpha,\|P_k\|_C/\beta\}  \right)
\end{align*}
Combining the previous observations, finally yields
\begin{align*}
\widehat{r}_J(u_{k+1})-\widehat{r}_J(u_{k}) \leq -s \Psi(u_k)+\frac{s^2L_f}{2} \|u_k-\widehat{u}_k\|^2_{L^2}.
\end{align*}
The claim in~\eqref{eq:optdescent} now follows immediately by minimizing the right-hand side w.r.t.~$s\in[0,1]$.
\end{proof}
In the following, Proposition~\ref{prop:descentlemma} is used to show that~$\Psi(u_k)$ asymptotically vanishes. Moreover, based on the next lemma, we also conclude that~$\sum^{N_k}_{j=1} \lambda^k_j$ approximates~$\TGV(u_k)$ for large~$k$.
\begin{lemma} \label{lem:boundforTVsum}
For all~$k \geq 1$ we have
\begin{align*}
0 \leq \sum^{N_k}_{j=1} \lambda^k_j-\TGV(u_k) \leq \Psi(u_k).
\end{align*}
\end{lemma}
\begin{proof}
Recall that~$\TGV(u_k)\leq M_0$ for all~$k\geq 1$ according to~\eqref{eq:TGVuk}. Hence we have
\begin{align*}
(\nabla f(u_k),u_k)_{L^2}+\TGV(u_k)-\min_{\TGV(v)\leq M_0} \left \lbrack (\nabla f(u_{k}),v)_{L^2}+\TGV(v)  \right \rbrack \geq 0
\end{align*}
Moreover, we readily calculate
\begin{align*}
(\nabla f(u_k),u_k)_{L^2}+\TGV(u_k)&-\min_{\TGV(v)\leq M_0} \left \lbrack (\nabla f(u_{k}),v)_{L^2}+\TGV(v) \right \rbrack
 \\
&= \TGV(u_k)-\sum^{N_k}_{j=1}\lambda^k_j-\min_{\substack{v \in B, \\m \in[0,M_0]}} \left \lbrack m\big((\nabla f(u_k),v)_{L^2}+1 \big) \right \rbrack \\ &=
\TGV(u_k)-\sum^{N_k}_{j=1}\lambda^k_j -\min_{m \in[0,M_0]} \left \lbrack m\big(1-\max\left\{\|p_k\|_C/\alpha,\|P_k\|_C/\beta\right\} \big) \right \rbrack  \\ &=
\TGV(u_k)-\sum^{N_k}_{j=1}\lambda^k_j + M_0\big(\max\left\{\|p_k\|_C/\alpha,\|P_k\|_C/\beta\right\}-1 \big)
\end{align*}
where we use~\eqref{eq:firstordersubprob2} in the first term of the left-hand side,~\eqref{eq:linearmin} as well as the positive one homogeneity of the~$\TGV$-function in the third and~$\max\left\{\|p_k\|_C/\alpha,\|P_k\|_C/\beta\right\}>1$ in the final equality. Combining both statements and reordering yields the claimed statement. 
\end{proof} 
\begin{theorem} \label{thm:convofpsi}
Let~$\{u_k\}_k$ be generated by Algorithm~\ref{alg:accgcg} and let Assumption~\ref{ass:convergence} hold. Then we have
\begin{align*}
\lim_{k \rightarrow \infty} \Psi(u_k)+ \left|\sum^{N_k}_{j=1} \lambda^k_j-\TGV(u_k) \right|=0.
\end{align*}
\end{theorem}
\begin{proof}
We only prove $\lim_{k \rightarrow \infty} \Psi(u_k)=0$, the result on~$\TGV(u_k)$ then follows from Lemma~\ref{lem:boundforTVsum}. We proceed similarly to~\cite{LacosteJulien16}. First, note that~$\{u_k\}_k$ and~$\{\widehat{u}_k\}_k$ are bounded in~$L^2(0,T)$ in virtue of~$J(u_k) \leq f(0)$ for all~$k\geq 1$, Assumption~\ref{ass:convergence}~(\textbf{A3}) and Proposition~\ref{prop:maxcertifies}. Moreover we have
\begin{align*}
0 \leq \sum^\infty_{j=1} \left \lbrack \widehat{r}_J(u_k)-\widehat{r}_J(u_{k+1}) \right\rbrack \leq \widehat{r}_J(u_1)
\end{align*}
and thus~$\lim_{k \rightarrow \infty}\left \lbrack \widehat{r}_J(u_k)-\widehat{r}_J(u_{k+1}) \right\rbrack=0 $.
Now, let~$\eps>0$ be arbitrary but fixed and let~$C>0$ be such that~$\max\{\|u_k\|_{L^2},M_0 \|\widehat{u}_k\|_{L^2}\}\leq C$,~$k\geq 1$. Furthermore, choose~$s_\eps \in [0,1]$ and~$K_\eps \in \mathbb{N}$ with
\begin{align*}
s_\eps \leq \frac{\eps}{4L_fC^2},\quad \widehat{r}_J(u_k)-\widehat{r}_J(u_{k+1}) \leq \frac{s_\eps \eps}{2} \quad \text{for all}~k \geq K_\eps.
\end{align*}
Inserting~$s_\eps$ into~\eqref{eq:descentforanys} we arrive at
\begin{align*}
0 \leq \Psi(u_k) \leq \frac{\widehat{r}_J(u_k)-\widehat{r}_J(u_{k+1})}{s_\eps} + s_\eps 2L_fC^2 \leq \eps \quad \text{for all}~k \geq K_\eps.
\end{align*}
Since~$\eps >0$ was chosen arbitrarily, we finally conclude~$\lim_{k\rightarrow \infty} \Psi(u_k)=0$.
\end{proof}
Next we show that~$\{u_k\}_k$ converges, on subsequences, towards stationary points of~\eqref{def:minprob}.  
\begin{coroll} \label{coroll:convofiterates}
The sequence~$\{u_k\}_k$ admits at least one weak* accumulation point in~$\BV(0,T)$. Every such point~$\bar{u}$ satisfies~$-\nabla f(\bar{u}) \in \partial_2 \TGV(\bar{u})$.
\end{coroll}
\begin{proof}
Since~$\BV(0,T)$ is the topological dual of a separable Banach space and~$\{u_k\}_k$ is bounded in~$\BV(0,T)$, the existence of a weak* convergent subsequence follows from the Banach-Alaoglu theorem. Fix an arbitrary weak* convergent subsequence, denoted by the same symbol, with limit~$\bar{u}$. Define
\begin{align*}
\bar{p}= \int^{\boldsymbol{\cdot}}_0 \nabla f(\bar{u})(x)~\mathrm{d}x,~\bar{P}= -\int^{\boldsymbol{\cdot}}_0 \bar{p}(x)~\mathrm{d}x.
\end{align*}
Since~$\BV(0,T) \hookrightarrow_c L^2(0,T)$ there holds~$u_k \rightarrow \bar{u}$ in~$L^2(0,T)$. Using this strong convergence as well as~Assumption~\ref{ass:convergence}~(\textbf{A2}) and \eqref{eq:firstordersubprob1} we get
\begin{align} \label{eq:ukconv1}
(\nabla f(\bar{u}),\ell)_{L^2}=\lim_{k\rightarrow \infty} (\nabla f(u_k),\ell)_{L^2}=0 \quad \text{for all}~\ell \in \Lc.
\end{align}
Thus,~$\bar{p} \in H^1_0(0,T)$ and~$\bar{P}\in H^2(0,T) \cap H^1_0(0,T)$. Moreover, we immediately get
\begin{align*}
p_k \rightarrow \bar{p},~P_k \rightarrow \bar{P} \quad \text{in } C_0(0,T)
\end{align*}
and
\begin{align*}
M_0(\max\{\|\bar{p}\|_C/\alpha,\|\bar{P}\|_C/\beta\}-1)= \lim_{k\rightarrow \infty} \Psi(u_k)=0.
\end{align*}
In view of Lemma~\ref{lem:normboundsonpandP}, this implies
\begin{align}\label{eq:eq:ukconv2}
-(\nabla f(\bar{u}),v)_{L^2} \leq \TGV(v) \quad \text{for all}~v \in \BV(0,T).
\end{align}
Finally, observe that
\begin{align*}
\TGV(\bar{u}) \leq \liminf_{k\rightarrow \infty} \TGV(u_k) \leq \liminf_{k\rightarrow \infty} \sum^{N_k}_{j=1} \lambda^k_j
\end{align*}
as well as
\begin{align*}
\liminf_{k\rightarrow \infty} \sum^{N_k}_{j=1} \lambda^k_j= \liminf_{k\rightarrow \infty} -(\nabla f(u_k),u_k)_{L^2}=-(\nabla f(\bar{u}),\bar{u})_{L^2} \leq \TGV(\bar{u}).
\end{align*}
using the lower semicontinuity of~$\TGV$ as well as~\eqref{eq:firstordersubprob2}.
This yields~$-(\nabla f(\bar{u}),\bar{u})_{L^2}=\TGV(\bar{u})$. Combining the previous observations we arrive at~$-\nabla f(\bar{u})\in \partial_2 \TGV(\bar{u})$ according to Proposition~\ref{prop:subdiffpre}.
\end{proof}
This section is finished by a~\textit{quantitative} description of the asymptotic behavior of~$\Psi(u_k)$. More in detail, it is shown that the smallest constraint violation encountered up to iteration~$k\geq 1$ vanishes at a rate of~$1/k^{1/2}$.
\begin{theorem} \label{thm:convrate}
Let~$\{u_k\}_k$ be generated by Algorithm~\ref{alg:accgcg} and let Assumption~\ref{ass:convergence} hold. Moreover, let~$C>0$ be such that $\max\{\|u_k\|_{L^2},M_0 \|\widehat{u}_k\|_{L^2}\}\leq C$ for all~$k\geq 1$. Then we have
\begin{align*}
\min_{1\leq j \leq k} \Psi(u_k)+ \min_{1\leq j \leq k} \left|\sum^{N_k}_{j=1} \lambda^k_j-\TGV(u_k) \right| \leq \sqrt{\frac{8L_f C^2 \widehat{r}_J(u_1)}{k}} \quad \text{ for all } k \geq 1.
\end{align*}
\end{theorem}
\begin{proof}
Again, we only prove the result for~$\Psi(u_k)$ and point out Lemma~\ref{lem:boundforTVsum}. First, invoking~\eqref{eq:optdescent} we arrive at
\begin{align} \label{eq:descenthelpineq}
\widehat{r}_J(u_{k+1}) - \widehat{r}_J(u_{k}) \leq - \max \left\{\frac{\Psi(u_k)^2}{8L_f C^2},\Psi(u_k)-2 C^2\right\}\leq -\frac{ \min_{1\leq j\leq k}\Psi(u_j)^2}{8L_f C^2}.
\end{align}
Summing up over~$1\leq j\leq$ on both sides we get
\begin{align*}
-\widehat{r}_J(u_{1}) \leq \widehat{r}_J(u_{k+1}) - \widehat{r}_J(u_{1}) \leq- k\frac{ \min_{1\leq j\leq k}\Psi(u_j)^2}{8L_f C^2}.
\end{align*}
Rearranging and taking the square root then yields the desired statement.\end{proof}
Moreover, if~$f$ is convex, there holds~$J(u_k)-\min_{u \in \BV(0,T)} J(u)= \mathcal{O}(1/k)$.
\begin{theorem}\label{thm:convrateconvex}
Let~$\{u_k\}_k$ be generated by Algorithm~\ref{alg:accgcg}, let Assumption~\ref{ass:convergence} hold and assume that~$f$ is convex. Then we have
\begin{align*}
J(u_k)-\min_{u \in \BV(0,T)} J(u) \leq \widehat{r}_J(u_k) \leq \Psi(u_k).\end{align*} 
Moreover, let~$C>0$ be such that $\max\{\|u_k\|_{L^2},M_0 \|\widehat{u}_k\|_{L^2}\}\leq C$ for all~$k\geq 1$. Then we have
\begin{align*}
J(u_k)-\min_{u \in \BV(0,T)} J(u) \leq \widehat{r}_J(u_k) \leq \frac{\widehat{r}_J(u_1)}{1+q(k-1)} \quad \text{where}~q=\frac{\widehat{r}_J(u_{1})}{8L_f C^2}.
\end{align*}
\end{theorem}
\begin{proof}
Let~$\bar{u}$ denote an arbitrary minimizer to~\eqref{def:minprob}. Then there holds~$\TGV(\bar{u})\leq M_0$ as well as
\begin{align*}
J(u_k)-\min_{u \in \BV(0,T)}  J(u)=J(u_k) - J(\bar{u}) \leq \widehat{r}_J(u_k)  
\end{align*}
By convexity of~$f$ we further have
\begin{align*}
\widehat{r}_J(u_k) &\leq (\nabla f(u_k),u_k-\bar{u})_{L^2}- \sum^{N_k}_{j=1} \lambda^k_j-\TGV(\bar{u}) =- (\nabla f(u_k),\bar{u})_{L^2}+\TGV(\bar{u})
\\ &\leq
- \min_{\TGV(u) \leq M_0} \left \lbrack (\nabla f(u_k),u)_{L^2}+\TGV(u) \right\rbrack.
\end{align*}
where we explicitly use~\eqref{eq:firstordersubprob2}. Now, invoking Proposition~\ref{prop:maxcertifies} yields
\begin{align*}
- \min_{\TGV(u) \leq M_0} \left \lbrack (\nabla f(u_k),u)_{L^2}+\TGV(u) \right\rbrack &= - \min_{m \in[0,M_0]} m \min_{v \in B}\left \lbrack (\nabla f(u_k),v)_{L^2}+1 \right\rbrack
\\&=  \max_{m \in[0,M_0]} \left \lbrack m (\max{\{\|p_k\|_C/\alpha, \|P_k\|_C/\beta\}}-1) \right \rbrack \\
&= \Psi(u_k).
\end{align*}
Thus,~\eqref{eq:descenthelpineq} can be further estimated by
\begin{align*}
\frac{\widehat{r}_J(u_{k+1})}{\widehat{r}_J(u_{1})} - \frac{\widehat{r}_J(u_{k})}{\widehat{r}_J(u_{1})}   \leq -\frac{\widehat{r}_J(u_{1})}{8L_f C^2} \left( \frac{\widehat{r}_J(u_{k})}{\widehat{r}_J(u_{1})} \right)^2
\end{align*}
The claimed convergence rate is now immediately deduced from~\cite[Lemma 5.1]{dunn}.
\end{proof}

\section{Deconvolution from restricted Fourier measurements}
\label{sec:deconvol}
The final section is devoted to illustrating the theoretical results as well as to highlighting the practical utility of Algorithm~\ref{alg:accgcg}. 
For this purpose, recall the definition of the Fourier transform of~$u\in L^2(0,T)$ extended by zero to $\R$ as 
\begin{align*}
(\mathcal{F}u)(\zeta )= \int^T_0 u(x)\operatorname{exp}(-i\zeta x)~\mathrm{d} x
\end{align*}
for every~$\zeta \in \R$. Given a finite number of frequencies~$\zeta_j \in \R$,~$j=1,\dots,M$, our interest lies in reconstructing an unknown piecewise affine signal~$u^\dagger$ from measurements of the form~$m^d_j:=(\mathcal{F}u)(\zeta_j )+\eps_j$,~$j=1,\dots,M$, where~$\eps_j$ denotes additive noise. This problem is severely ill-posed due to the gap between the infinite-dimensional nature of the unknown signal and the limited availability of measurements. For example, consider~$T=10$ and the ground truth~$u^\dagger$ depicted in Figure~\ref{fig:groundill}. Its Fourier transform~$\mathcal{F}u^\dagger$ is sampled at eight equidistant points in $(0,T)$, see Figure~\ref{fig:setup}. Then there are~$\bar{a}_j, \bar{b}_j \in \R$,~$j=1,\dots,M$, such that the function
\begin{align} \label{eq:tildeu}
\tilde{u}(x)=  \sum^M_{j=1} \bar{a}_j \cos(\zeta_j x)-\bar{b}_j \sin(\zeta_j x)
\end{align}
satisfies~$(\mathcal{F}\tilde{u})(\zeta_j)=(\mathcal{F}u^\dagger) (\zeta_j)$ for all~$j=1,\dots,M$. Note that~$\tilde{u}$ does not share~\textit{any} structural features of~$u^\dagger$  and even has a period smaller than $T$ due to the location of the first frequency measurement.
\begin{figure}[htb]
\begin{subfigure}[t]{.3\linewidth}
\centering
\includegraphics[scale=0.4]{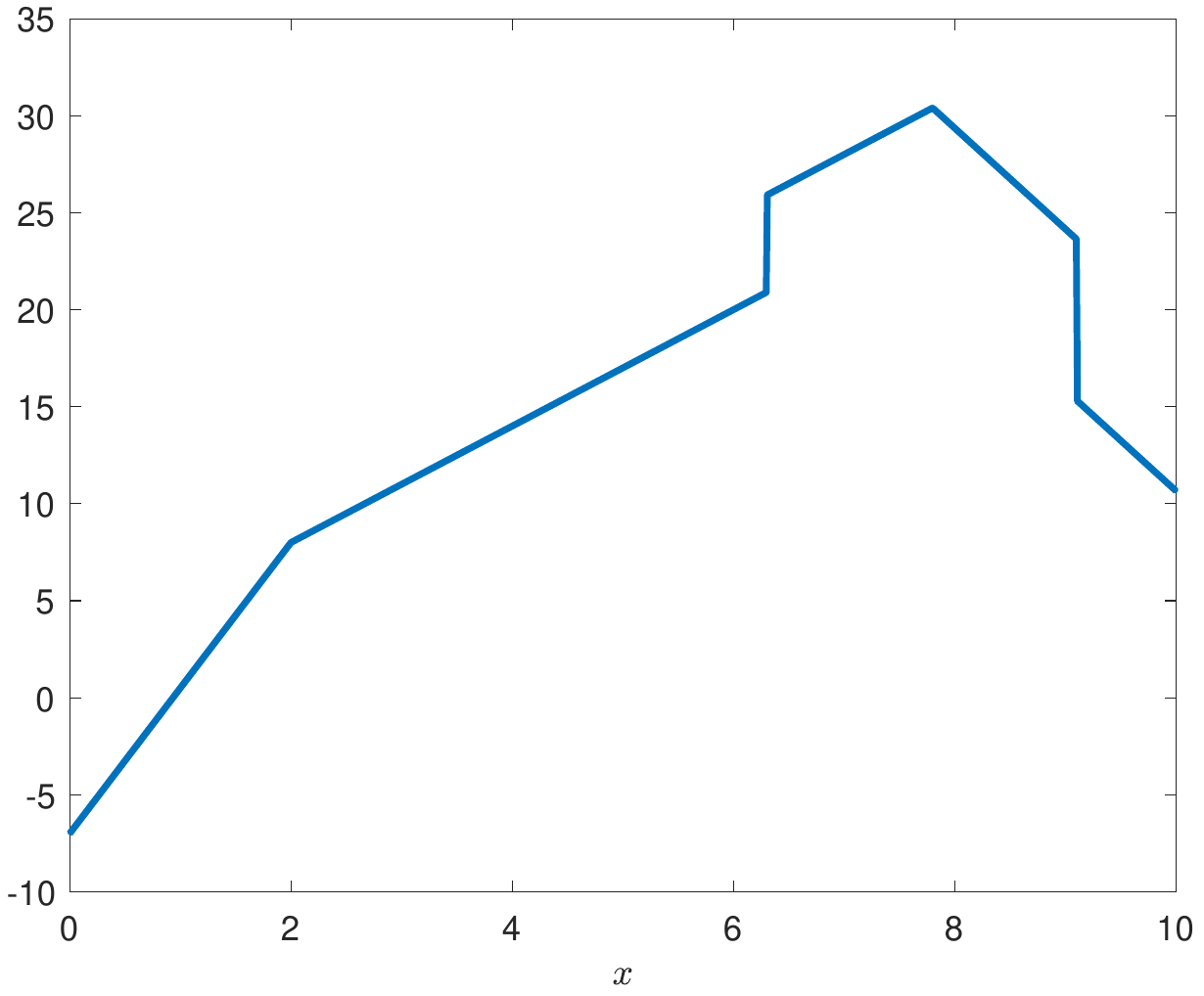}
\caption{Ground truth~$u^\dagger$.}
\label{fig:groundill}
\end{subfigure}
\quad
\begin{subfigure}[t]{.3\linewidth}
\centering
\includegraphics[scale=0.4]{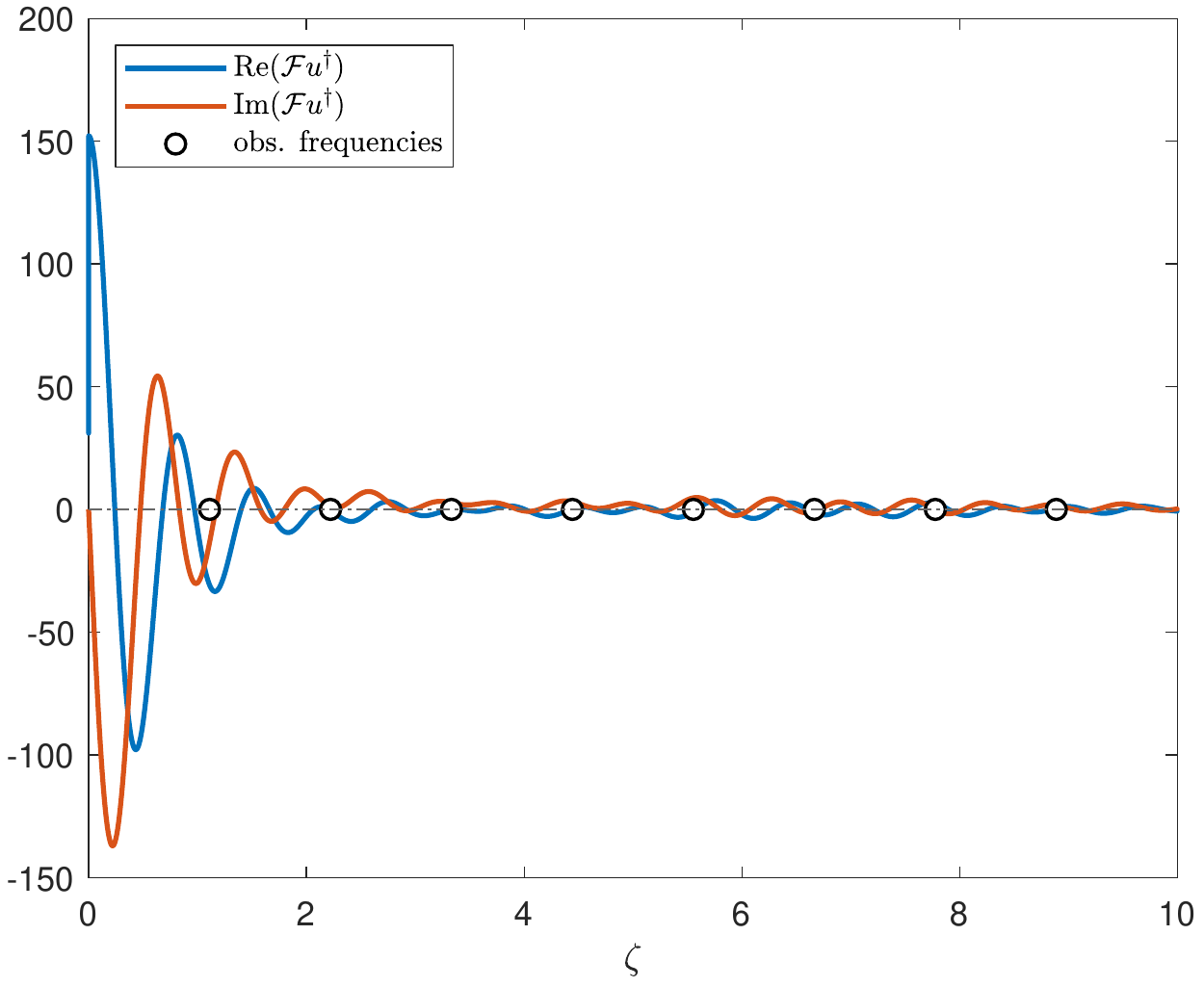}
\caption{$\mathcal{F}u^\dagger$ and~$\zeta_j$.}
\label{fig:setup}
\end{subfigure}
\quad
\begin{subfigure}[t]{.3\linewidth}
\centering
\includegraphics[scale=0.4]{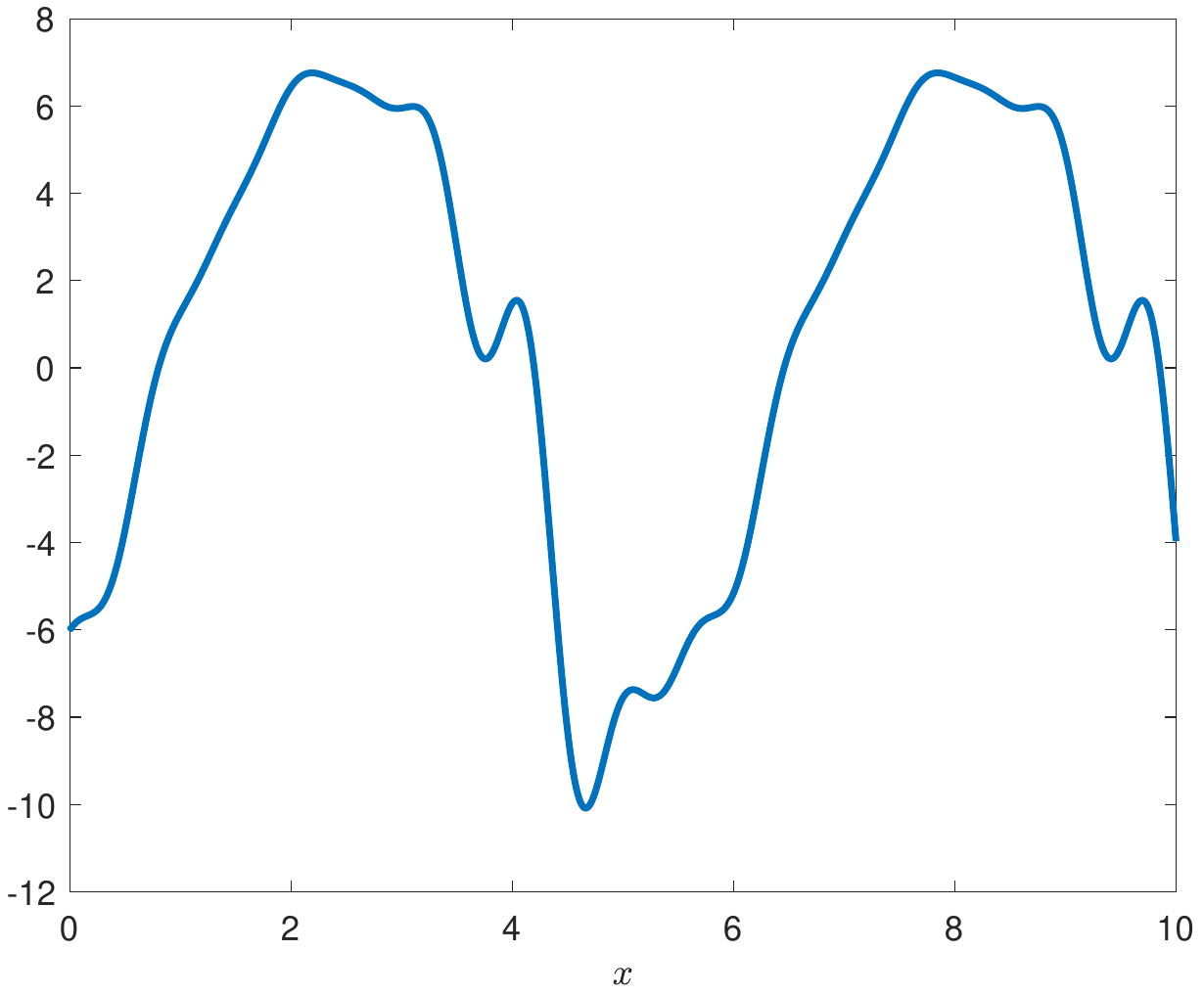}
\caption{Function~$\tilde{u}$ from~\eqref{eq:tildeu}.}
\label{fig:tildeu}
\end{subfigure}
\caption{Inverse problem setup.}
\label{fig:invprobsetup}
\end{figure}
In order to preserve the piecewise affine linear structure of~$u^\dagger$ in the reconstruction and to alleviate the ill-posedness of this inversion, we thus propose to consider the~$\TGV$-regularized problem
\begin{align} \label{def:deconvproblem}
\min_{u \in \BV(0,T)} J(u) \coloneqq\left \lbrack \frac{1}{2}\sum^M_{j=1} |(\mathcal{F}u)(\zeta_j)-m^{d}_{j}|^2_{\mathbb{C}}+ \TGV(u) \right \rbrack. \tag{P}
\end{align}
For abbreviation, define
\begin{align*}
F \colon \mathbb{C}^M \to \R,~ F(m)= \frac{1}{2}\sum^{M}_{j=1} |m_j-m^{d}_{j}|^2_{\mathbb{C}},
\end{align*}
as well as the linear and continuous observation operator
\begin{align} \label{def:observoperator}
H \colon L^2(0,T) \to \mathbb{C}^M, \quad Hu= \left((\mathcal{F}u)(\zeta_1),\dots, (\mathcal{F}u)(\zeta_M) \right)^\top.
\end{align}
The following proposition addresses the existence of sparse minimizers to~\eqref{def:deconvproblem}.
\begin{prop} \label{prop:deconvex}
Let~$H_{|\mathcal{L}}$ be injective. Then~$f= F\circ H$ satisfies Assumption~\ref{ass:convergence} and~\eqref{def:deconvproblem} admits a minimizer~$\bar{u}$ of the form~\eqref{def:sparseuintro} with~$N_K+N_S \leq 2M-2$. Moreover, the optimal observation is unique, i.e., for two solutions~$\bar{u}_1,~\bar{u}_2$ to~\eqref{def:deconvproblem} we have~$H\bar{u}_1=H\bar{u}_2$.\end{prop}
\begin{proof}
Assumptions~\ref{ass:convergence} (\textbf{A1})-(\textbf{A3}) can be verified by direct calculations. In order to show (\textbf{A4}), let~$B>0$ and~$u \in \BV(0,T)$ with~$J(u) \leq B$  be arbitrary but fixed.
Then there holds
\begin{align*}
\max \left\{\big|Hu-m^d\big|_{\mathbb{C}^M}, \TGV(u)\right\} \leq C_B\end{align*}
for some~$C_B>0$ depending on~$B$ but not on~$u$. In the following~$c>0$ denotes a generic constant which is independent of~$u$. We start by estimating
\begin{align*}
\|u\|_{\BV} \leq c \big(\|u\|_{L^1}+ \TGV(u)\big) \leq c \big(\|u\|_{L^2}+ \TGV(u)\big) \leq c \big(\|u\|_{L^2}+ C_{B}\big).
\end{align*}
Denoting by~$\mathcal{L}^\top$ the orthogonal complement of~$\mathcal{L}$ in~$L^2(0,T)$ and by~$P_{\mathcal{L}},~P_{\mathcal{L}^\top}$, the respective orthogonal projections, we have
\begin{align*}
\|u\|_{L^2}=\|P_{\mathcal{L}}u\|_{L^2}+\|P_{\mathcal{L}^\top}u\|_{L^2} ,~\|P_{\mathcal{L}^\top}u\|_{L^2}=\|u+\mathcal{L}\|_{L^2/\mathcal{L}}.
\end{align*}
Now, estimating as in the proof of Proposition~\ref{prop:quotientstuff} and utilizing the continuous embedding~$\BV(0,T) \hookrightarrow L^2(0,T)$, there holds
\begin{align*}
\|P_{\mathcal{L}^\top}u\|_{L^2}=\|u+\mathcal{L}\|_{L^2/\mathcal{L}} \leq c \|u+\mathcal{L}\|_{\BV/\mathcal{L}} \leq c \TGV(u)\leq c\,C_B.
\end{align*}
Moreover, using the injectivity of~$H$ on~$\mathcal{L}$, we get
\begin{align*}
\|P_{\mathcal{L}}u\|_{L^2} &\leq c \left( |HP_{\mathcal{L}}u|_{\mathbb{C}^M} \right) \leq c \left( |HP_{\mathcal{L}^\top}u|_{\mathbb{C}^M}+ |Hu-m^d|_{\mathbb{C}^M}+|m^d|_{\mathbb{C}^M} \right) 
\\& \leq  c \left( \TGV(u)+ |Hu-m^d|_{\mathbb{C}^M}+|m^d|_{\mathbb{C}^M} \right) 
\\& \leq  c \left( 2C_B+|m^d|_{\mathbb{C}^M} \right).
\end{align*}
Hence, combining both estimates, we conclude that the sublevel set~$\{\,u \in \BV(0,T)\;|\;J(u)\leq B\,\}$ is bounded. Since~$B>0$ was chosen arbitrary,~Assumption~\ref{ass:convergence} (\textbf{A4}) follows. Thus, according to Proposition~\ref{prop:existence}, Problem~\eqref{def:deconvproblem} admits at least one minimizer. Invoking Theorem~\ref{thm:convexrepresenter} then yields the existence of minimizer of the form~\eqref{def:sparseuintro} with~$N_k+N_S \leq 2M-2$ since~$\mathbb{C}^M \simeq \R^{2M}$ and~$\operatorname{dim}(H(\mathcal{L}))=2$. The uniqueness of the optimal observation follows immediately from the strict convexity of~$F$.  
\end{proof}
In the following, the optimal observation is denoted by~$\bar{m}$. While Proposition~\ref{prop:deconvex} ensures the existence of at least one minimizer in the form~\eqref{def:sparseuintro}, Corollary~\ref{coroll:analytic} implies that~\textit{all} solutions to~\eqref{def:deconvproblem} are sparse if the optimal misfit is nonzero.
\begin{coroll}
Assume that~$\bar{m} \neq m^d$. Then every solution to~\eqref{def:deconvproblem} is of the form~\eqref{def:sparseuintro}. 
\end{coroll}
\begin{proof}
Let~$\bar{u}$ be an arbitrary minimizer to~\eqref{def:deconvproblem} and let~$\bar{m}$ denote the unique optimal observation. We readily verify that~$\nabla f(\bar{u})$, and thus also~$\bar{p}$ and~$\bar{P}$, \eqref{def:pandPopt}, are analytic on~$(0,T)$ due to
\begin{align*}
\nabla f(\bar{u})(x)= \sum^M_{j=1} \operatorname{Re}(\bar{m}_j-m^d_j) \cos(\zeta_j x)-\sum^M_{j=1} \operatorname{Im}(\bar{m}_j-m^d_j) \sin(\zeta_j x).
\end{align*}
Moreover there holds~$\nabla f(\bar{u})\neq 0$ since the functions~$\cos(\zeta_j x),~\sin(\zeta_j x)$,~$j=1,\dots,M$, are linearly independent and~$\bar{m}-m^d \neq 0$. The claimed statement now follows immediately from Corollary~\ref{coroll:analytic}.
\end{proof}

In order to illustrate the advantages of solving Problem~\eqref{def:deconvproblem}, we return to the setup described in Figure~\ref{fig:invprobsetup}, i.e., we set
\begin{align*}
u^\dagger&=-4.5 K_2+ 5 S_{6.3}-8.2 K_{7.8}-8.3 S_{9.1}+  3x+2\\&=\bar{\mu}_1 K_2+ \bar{\mu}_2 S_{6.3}+\bar{\mu}_3 K_{7.8}+\bar{\mu}_4 S_{9.1}+  3x+2,
\end{align*}
see Figure~\ref{fig:groundill}, and choose~$T=10$ as well as $\zeta_j=j\cdot (10/9)  $,~$j=1,\dots,8$, as depicted in Figure~\ref{fig:setup}. Subsequently, a measurement vector~$m^d \in \mathbb{C}^8$ is generated by setting~$m^d =H u^\dagger+ \eps$ where~$\eps \in \mathbb{C}^8$ is a noise vector with~$\|\eps\|_{\mathbb{C}^8}/\|Hu^\dagger\|_{\mathbb{C}^8}=0.1$. The regularization parameters are chosen empirically as~$\alpha=2.205$ and~$\beta=2.5344$. By direct computation, it is readily verified that~$H$ is injective on~$\mathcal{L}$ in this case. Hence, Problem~\eqref{def:deconvproblem} admits a solution which can be computed using Algorithm~\ref{alg:accgcg}.
\subsection{Practical realization of Algorithm~\ref{alg:accgcg}}
Let us briefly discuss the practical realization of Algorithm~\ref{alg:accgcg} without discretizing the interval. For this purpose, we point out that~$ HK_x$,~$HS_x$ and~$H \ell$ admit a closed-form representation for every~$x \in (0,T)$,~$\ell \in \mathcal{L}$. Hence, given the k-th iterate~$u_k \in \operatorname{cone}(\mathcal{A}_k)+\mathcal{L}$, the measurements~$Hu_k$, the corresponding derivative
\begin{align*}
\nabla f(u_k)(x)= \sum^M_{j=1} \operatorname{Re}\big((Hu_k)_j-m^d_j\big) \cos(\zeta_j x)-\sum^M_{j=1} \operatorname{Im}\big((Hu_k)_j-m^d_j\big) \sin(\zeta_j x).
\end{align*}
as well as dual variables~$p_k,P_k$, see~\eqref{eq:iterduals}, can be computed without numerical approximation. Every iteration of Algorithm~\ref{alg:accgcg} then requires the solution of two subproblems. First, we have to determine global extrema of~$p_k$ and~$P_k$ in step 3. This is done by computing solutions of~$p'_k(x)=0$ and~$P'_k(x)=0$ using a Newton method starting at equally spaced points~$x^j_0= j  \cdot 0.1$,~$j=1,\dots,99$. Then~$\hat{x}^S_k,~\hat{x}^K_k \in (0,T) $ are chosen from the set of computed solutions by comparing the corresponding function values. Second, step 1 and 6, require the solution of the finite dimensional convex Problem~$(\mathcal{P}(\mathcal{A}^+_k))$. For this, we rely on a semismooth Newton method for the "normal map" reformulation of its first order sufficient optimality conditions. In each iteration the method is warmstarted using the current iterate to construct a good starting point. Moreover we further enhance its practical performance by incorporating a heuristic globalization strategy based on damped Newton steps. The method is run for a maximum of~$100$ iterations or until the constraint violation
\begin{align*}
\Psi(u_k)= \frac{\|m^d\|^2_{\mathbb{C}^8}}{2} \big( \max\left\{\|p_k\|_C/\alpha,\|P_k\|_C/\beta\right\}-1\big)
\end{align*}
satisfies~$\Psi(u_{\bar{k}})\leq 10^{-10}$ for some~$\bar{k}\leq 100$.
Note that, since~\eqref{def:deconvproblem} is convex, we also have
\begin{align*}
J(u_{\bar{k}})- \min_{u \in \BV(0,T)} J(u) \leq \widehat{r}_J(u_{\bar{k}})\leq \Psi(u_{\bar{k}})\leq 10^{-10}
\end{align*}
in virtue of Theorem \ref{thm:convrateconvex}. All of the following computations were carried out on Matlab 2019 on a notebook with~$32$ GB RAM and an Intel\textregistered Core\texttrademark ~i7-10870H CPU@2.20 GHz.

\subsection{Qualitative features of the reconstruction} \label{subsec:features}
Starting from~$\mathcal{A}_0=\{S_{7.5}\}$ and~$u_0=S_{7.5}$, the method stops after~$\bar{k}=50$ iterations since~$\Psi(u_{\bar{k}})\leq 10^{-10}$. In Figure~\ref{fig:sol} we plot the computed function~$\bar{u}=u_{\bar{k}}$ alongside the ground truth~$u^\dagger$. We point out that~$\bar{u}$ is piecewise affine linear and its kinks and jumps closely approximate those of~$u^\dagger$. At a first glance, it seems that both~$\bar{u}$ and~$u^\dagger$ share the same number of jumps/kinks. Upon a closer inspection, though, we observe that the kink at approximately $7.8$ is approximated by two kinks of equal sign in~$\bar{u}$. We assume that this ``clustering'' behavior is a consequence of numerical rounding errors since the positions of both jumps/kinks only differ by a magnitude of order~$10^{-7}$. Moreover, replacing both jumps/kinks by a single one of the combined magnitude at either of both candidate locations only introduces a relative error of~$10^{-16}$ in the objective functional value. In any case, the combined number of kinks and jumps in~$\bar{u}$ is considerably smaller than the theoretical upper bound of~$2M-2=14$ predicted by Proposition~\ref{prop:deconvex}. Next, the reconstruction of the coefficients is investigated. For a better visualization, we plot a Dirac delta with weight~$\bar{\mu}_j$,~$j=1,\dots,4$, at the position of the corresponding kink/jump (dashed lines for kinks, solid line for jumps) in Figure~\ref{fig:coefficients}. The recovered function is treated analogously with the exception of merging the aforementioned clustering kinks into one of combined magnitude.    \begin{figure}[htb]
\begin{subfigure}[t]{.49\linewidth}
\centering
\includegraphics[scale=0.5]{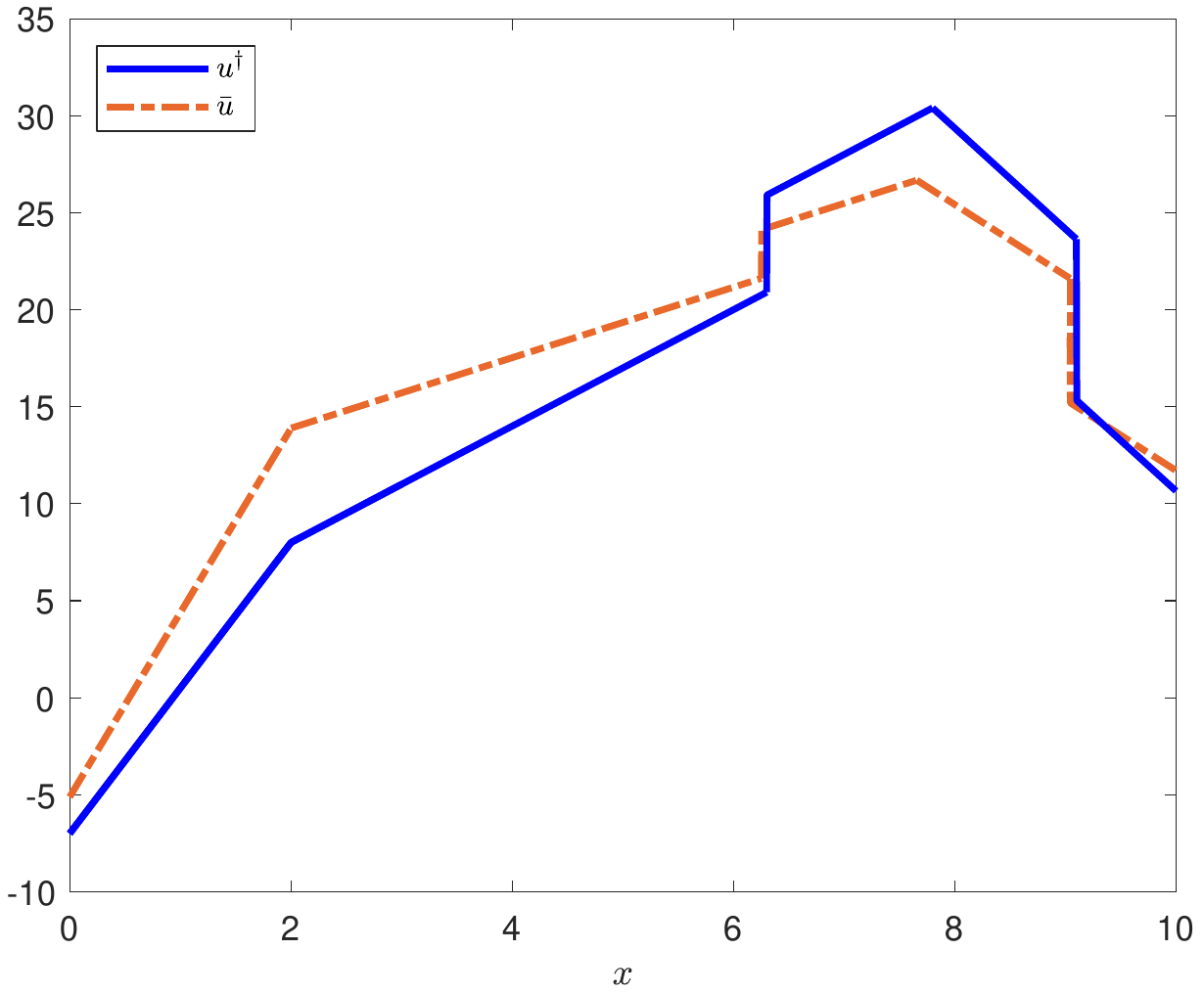}
\caption{Ground truth~$u^\dagger$ and~reconstruction~$\bar{u}$.}
\label{fig:sol}
\end{subfigure}
\quad
\begin{subfigure}[t]{.49\linewidth}
\centering
\includegraphics[scale=0.5]{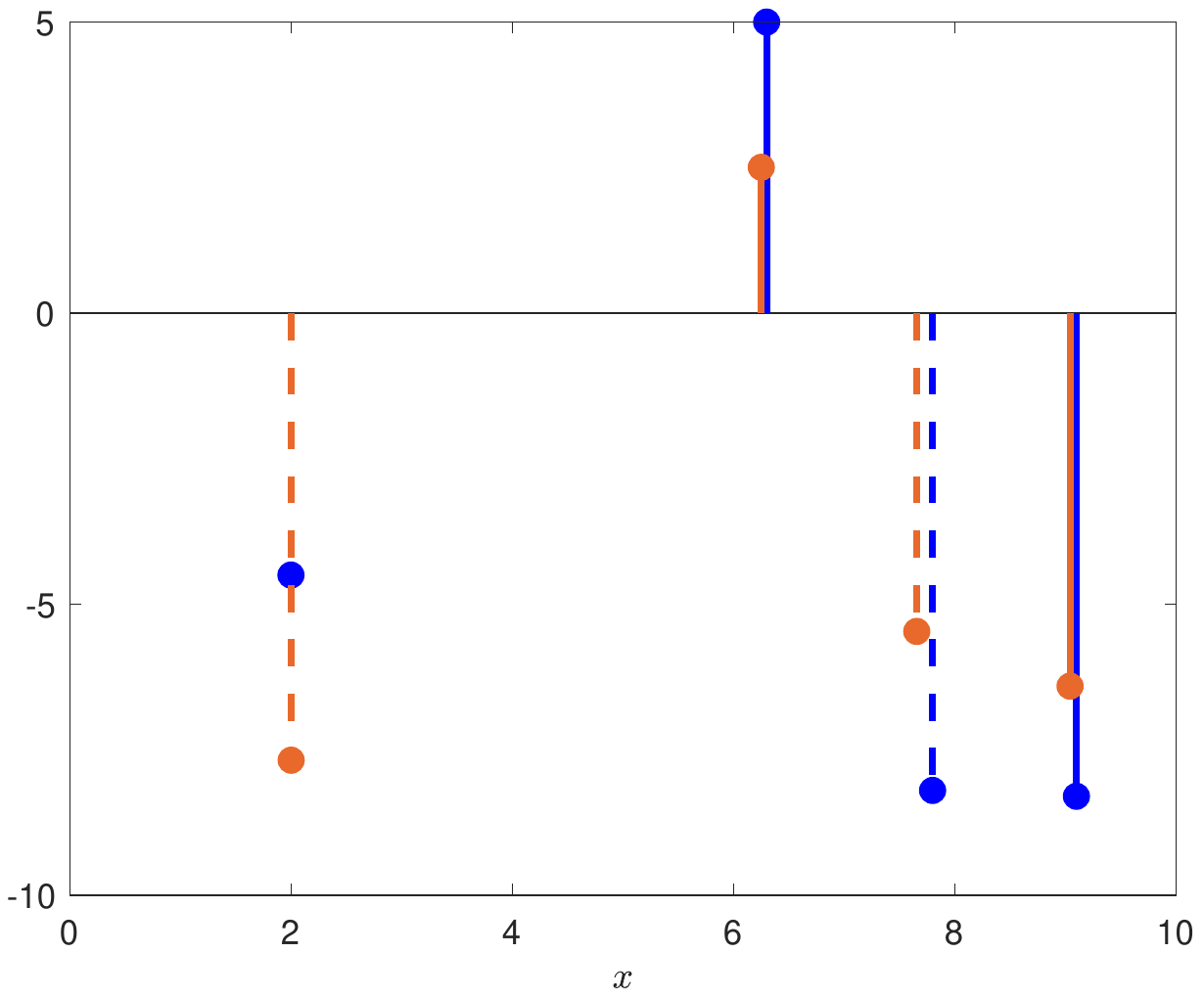}
\caption{Reference (blue) and computed (orange) coefficients.}
\label{fig:coefficients}
\end{subfigure}
\caption{Ground truth vs. reconstructions}
\label{fig:references}
\end{figure}
Observe that by solving~\eqref{def:deconvproblem} we successfully reconstruct the sign of every~$\bar{\mu}_j$ and provide decent estimates of its magnitude (after clustering). We also point out that, with the exception of the leftmost kink, the recovered coefficient is smaller than the reference. A possible explanation of this behavior can be found in Theorem~\ref{thm:finitextremasets} which states that the absolute values of the recovered coefficients are given by the solution to a minimization problem with an~$\ell_1$-type regularization term. For this type of problem, underestimation bias is a well-established phenomenon.
\subsection{Optimality conditions and performance of Algorithm 1}
We take the opportunity and verify the structural properties of~$\bar{u}$ derived in Theorem~\ref{thm:finitextremasets}. For this purpose, the dual variables~$\bar{p}$ and~$\bar{P}$ are plotted in Figures~\ref{fig:smallp} and~\ref{fig:largeP}. Function values that correspond to the positions of jumps in~$\bar{u}$ are marked by orange dots in Figure~\ref{fig:smallp}. The same is done for function values corresponding to jumps in Figure~\ref{fig:largeP}.
\begin{figure}[htb]
\begin{subfigure}[t]{.49\linewidth}
\centering
\includegraphics[scale=0.5]{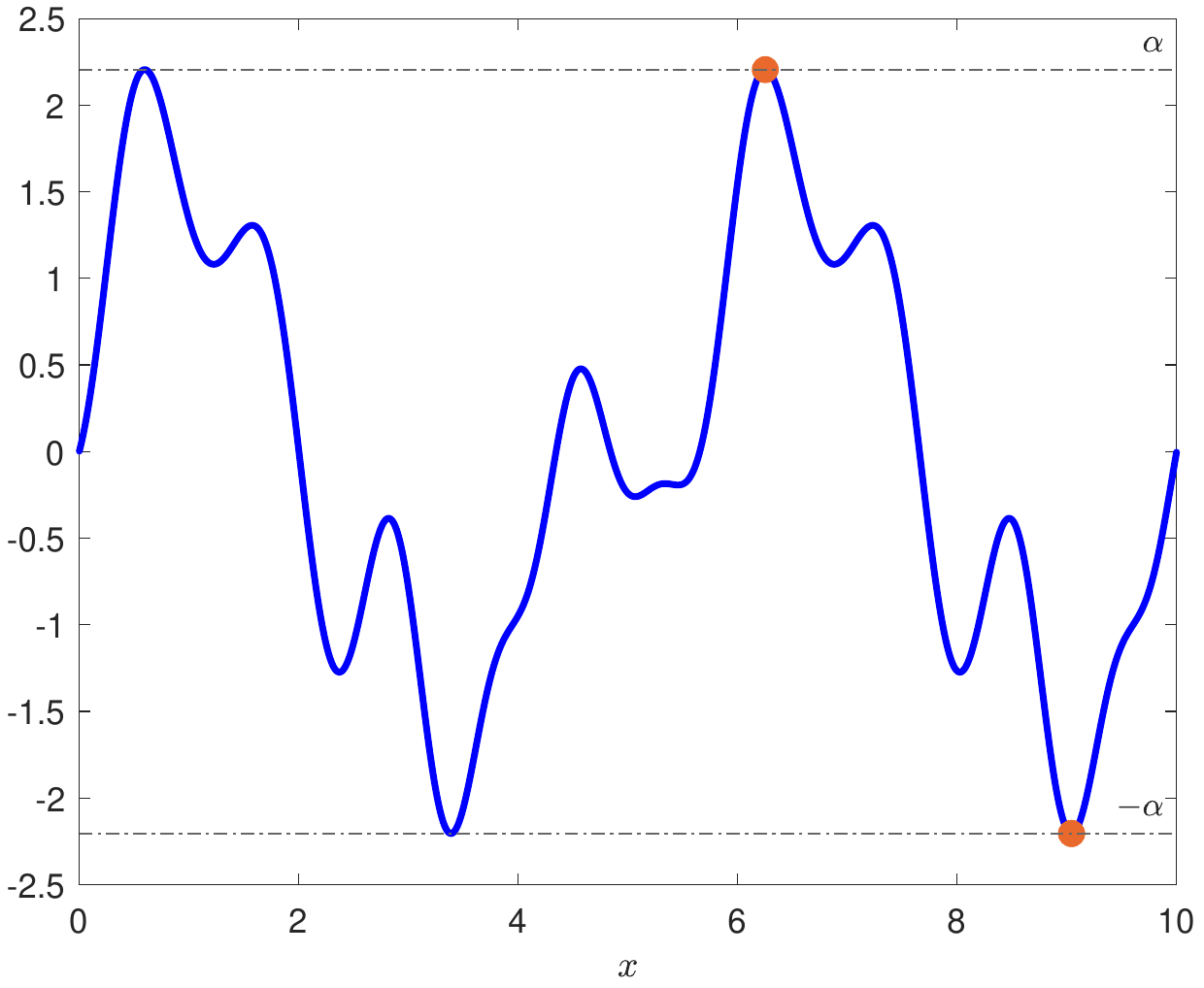}
\caption{First primitive~$\bar{p}$.}
\label{fig:smallp}
\end{subfigure}
%\quad
%\begin{subfigure}[t]{.49\linewidth}
%\centering
%\includegraphics[scale=0.5]{smallphess.pdf}
%\caption{Second derivative~$\bar{p}''$.}
%\label{fig:smallphess}
%\end{subfigure}
\quad
\begin{subfigure}[t]{.49\linewidth}
\centering
\includegraphics[scale=0.5]{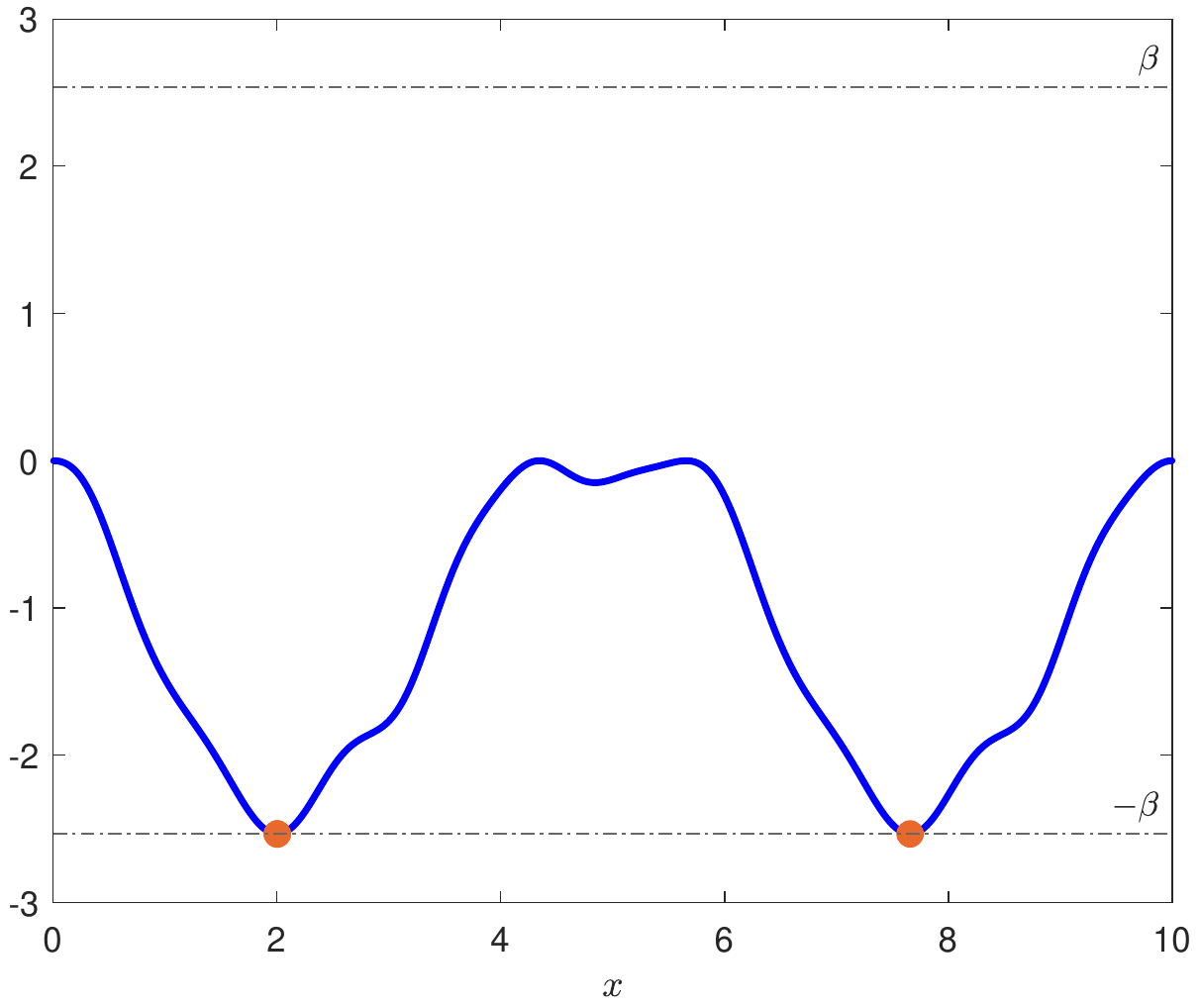}
\caption{Second primitive~$\bar{P}$.}
\label{fig:largeP}
\end{subfigure}
\caption{Dual variables.}
\label{fig:dual}
\end{figure}   
First, note that~$\bar{u}$ is nonzero. Hence, Theorem~\ref{thm:necessaryopt} implies that~$\|\bar{p}\|_C=\alpha$,~$\|\bar{P}\|_C=\beta$. This is readily verified from the plots. Furthermore, the set of global extremas for each function consists of finitely many points. Consequently, the positions of jumps and kinks of~$\bar{u}$ align with minimizers/maximizers of~$\bar{p}$ and~$\bar{P}$, respectively, as predicted by Theorem~\ref{thm:finitextremasets}. Moreover the sign of the recovered coefficients for every jump/kink coincides with~$\bar{p}/\alpha$ and $\bar{P}/\beta$, respectively, as expected.
%Second, in order to further investigate the clustering of jumps/kinks in the reconstructed function~$\bar{u}$, we also plot the second derivatives of~$\bar{p}$ and~$\bar{P}$ in Figures~\ref{fig:smallphess} and~\ref{fig:largePhess}. Again, function values associated to the jumps and kinks of~$\bar{u}$ are marked by orange dots. These plots suggest that the curvature of the dual variables does not degenerate around the recovered jumps and kinks. Hence, the corresponding extrema of~$\bar{p}$ and~$\bar{P}$ are isolated. This affirms our earlier guess that the observed clustering is due to numerical inaccuracies. 
\begin{figure}[htb]
\begin{subfigure}[t]{.49\linewidth}
\centering
\includegraphics[scale=0.5]{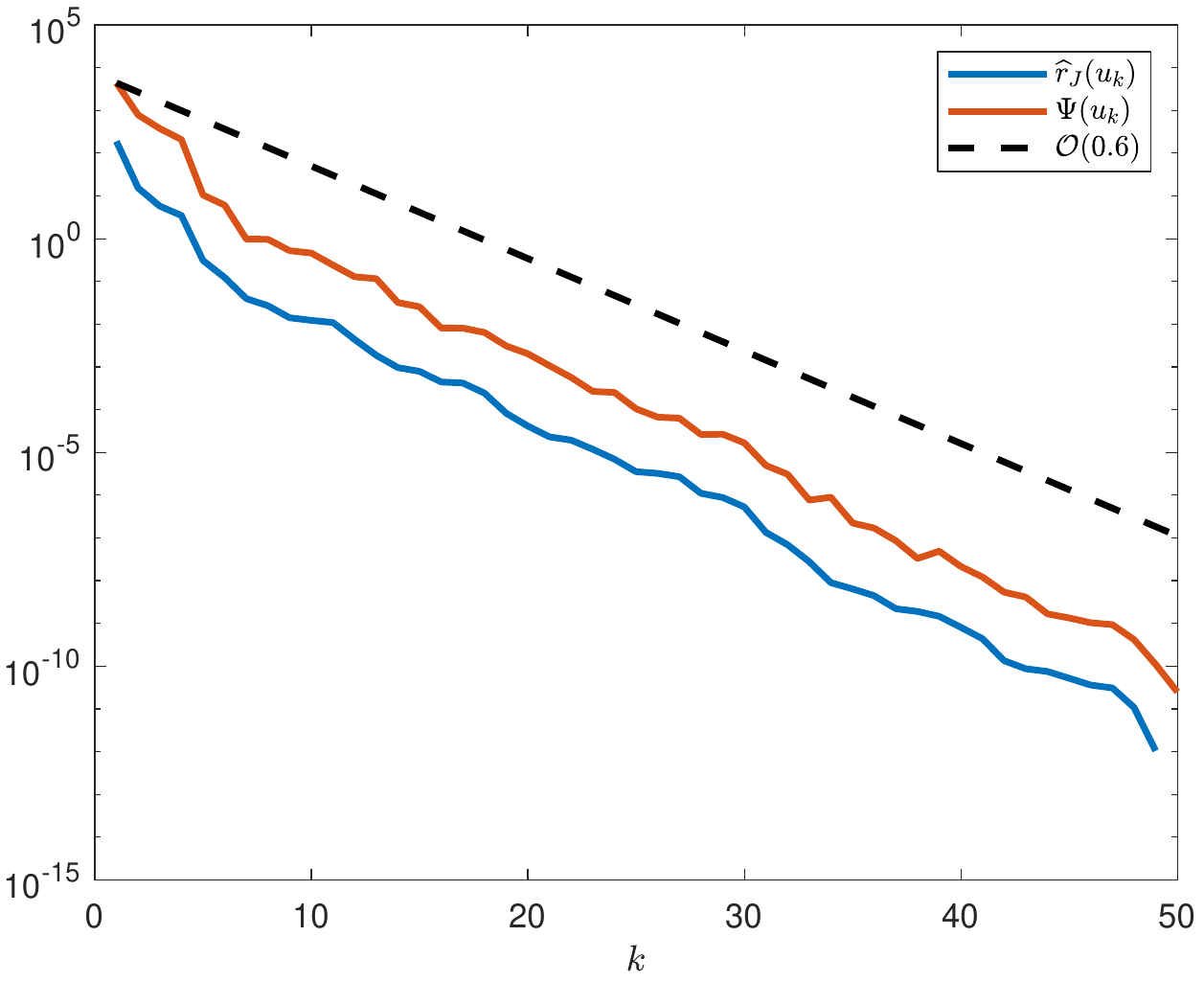}
\caption{$\widehat{r}_J(u_k)/\Psi(u_k)$ vs.~$k$.}
\label{fig:convrate}
\end{subfigure}
%\quad
%\begin{subfigure}[t]{.49\linewidth}
%\centering
%\includegraphics[scale=0.5]{smallphess.pdf}
%\caption{Second derivative~$\bar{p}''$.}
%\label{fig:smallphess}
%\end{subfigure}
\quad
\begin{subfigure}[t]{.49\linewidth}
\centering
\includegraphics[scale=0.5]{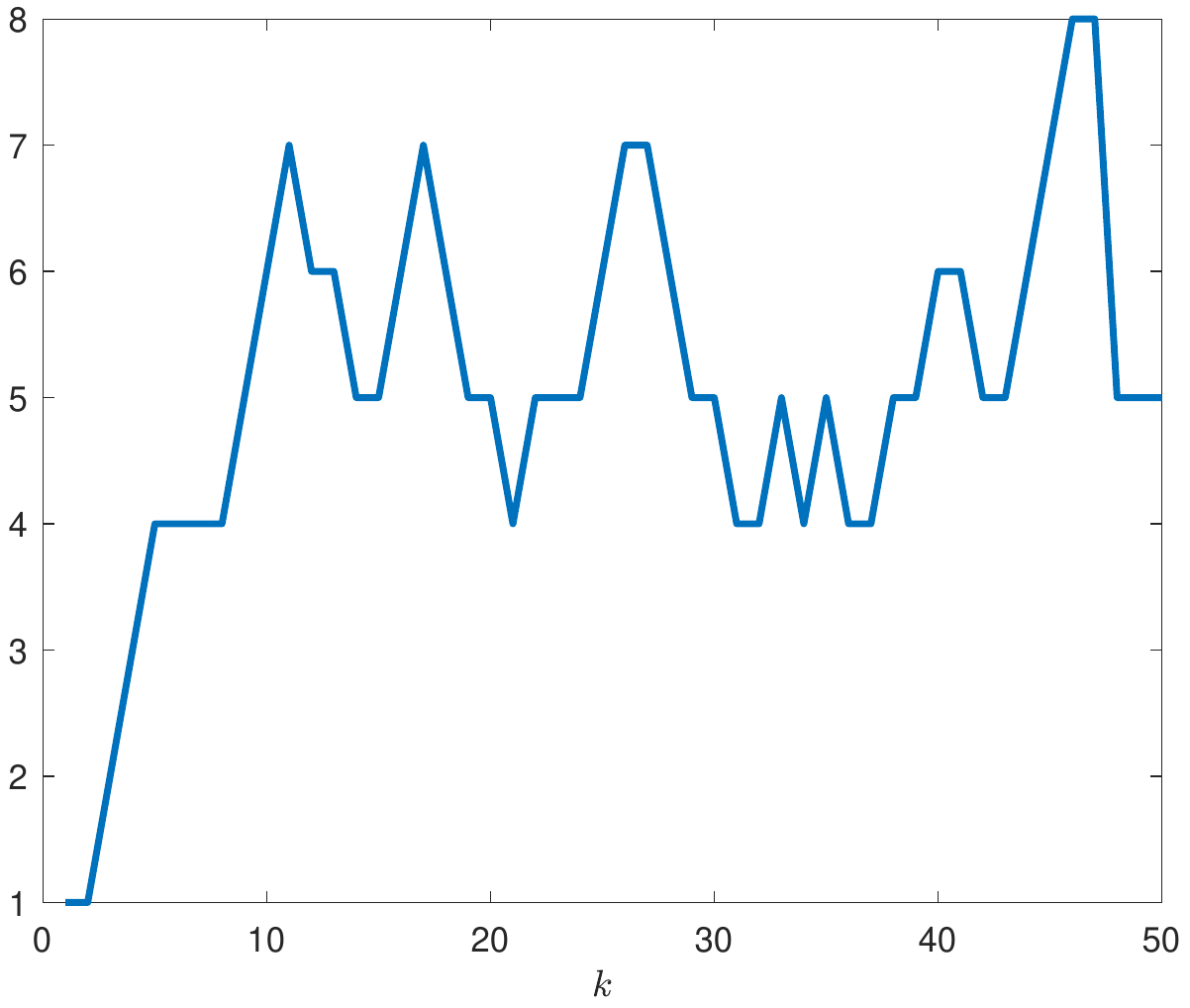}
\caption{$\#\mathcal{A}_k$ vs.~$k$.}
\label{fig:activesize}
\end{subfigure}
\caption{Evolution of relevant quantities.}
\label{fig:evol}
\end{figure}  
Finally, we plot the convergence of the constraint violation~$\Psi(u_k)$ as well as of the residual
\begin{align*}
\widehat{r}_J(u_k)= f(u_k)+\sum^{N_k}_{j=1} \lambda^k_j- \min_{u \in \BV(0,T)} J(u) \approx f(u_k)+\sum^{N_k}_{j=1} \lambda^k_j-f(u_{\bar{k}})+\sum^{N_{\bar{k}}}_{j=1} \lambda^{\bar{k}}_j. 
\end{align*} 
This approximation is justified by~$\widehat{r}_J(u_{\bar{k}})\leq 10^{-10}$. The resulting graphs can be found in Figure~\ref{fig:convrate}. Clearly, these computations suggest a linear rate of convergence for both quantities which is vastly better than the~$\mathcal{O}(1/k^{1/2})$ and~$\mathcal{O}(1/k)$ behavior predicted by Theorem~\ref{thm:convrate} and~\ref{thm:convrateconvex}, respectively. Alongside, in Figure~\ref{fig:activesize}, we plot the evolution of the active set size. Most remarkably,~$\#\mathcal{A}_k$ does not exceed~$8$ throughout the iterations. This highlights the importance of the point removal step in Algorithm~\ref{alg:accgcg}. Moreover, keeping~$\#\mathcal{A}_k$ small in combination with a warm-started semismooth Newton method leads to a fast resolution of the subproblems~$(\mathcal{P}(\mathcal{A}^+_k))$. In fact, the overall computation time of Algorithm~\ref{alg:accgcg} for the present example is merely $4.5$ s.    
In summary, these observations show the practical efficiency of methods in the spirit of Algorithm~\ref{alg:accgcg} for TGV-regularization. Of course, they also spark the demand of further research concerning e.g. an improved quantitative convergence analysis of Algorithm~\ref{alg:accgcg}. However, since the focus of the present manuscript lies on the characterization of the extremal points and immediate consequences thereof, thorough investigations are postponed to a follow-up work.  
\appendix
\appendix
\section{A counterexample} \label{sec:counterexample}
In this appendix we finally give the counterexample announced in Remark~\ref{rem:altsubrpob}. More in detail, a set~$\mathcal{A}=\{u_j\}^N_{j=1}$ in~$\ext B$ as well as a convex and smooth function~$f \colon L^2(0,T) \to \R$ with~$\min \eqref{eq:altsubprob} < \min \eqref{def:subprob}$
are constructed. For this purpose, choose~$T=10$ and~$\alpha, \beta>0$ with~$\beta/\alpha=1/2$. Moreover set~$\bar{x}_1=6$,~$\bar{x}_2=6.25$ and fix~$\bar{\lambda}_1,\bar{\lambda}_2>0$ such that
\begin{align*}
|\bar{\lambda}_2-\bar{\lambda}_1|< 0.25,~ 3.75 < \bar{\lambda}_1+\bar{\lambda}_2.
\end{align*}
Under this prerequisites, the function~$\bar{u}=\bar{\lambda}_1 u_1+\bar{\lambda}_2 u_2$ with~$u_1=K_{\bar{x}_1}/\beta$ and~$u_2=-K_{\bar{x}_2}/\beta$ satisfies~$\TGV(\bar{u})< \bar{\lambda}_1+\bar{\lambda}_2$, see Remark~\ref{rem:counterexample}.
The following proposition provides an outline for the construction of a counterexample.
\begin{prop} \label{s}
Assume that there is~$\bar{\varphi} \in L^2(0,T) $ such that
\begin{align} \label{eq:condonphi}
(u_j,\bar{\varphi})_{L^2}+1=0 \quad \text{for all}~ j=1,2 ,~(\ell,\bar{\varphi})_{L^2}=0 \quad \text{for all}~\ell \in \Lc.
\end{align}
Then~$(\bar{\lambda}_1,\bar{\lambda}_2,0) \in \R^2_+ \times \Lc$ is a minimizer of 
\begin{align} \label{probcounter}
\min_{\lambda \in \R^2_+,\ell \in \Lc} \left \lbrack \frac{1}{2} \int^{10}_0 \Big|\sum^2_{j=1}\lambda_ju_j(x)-u_d(x)\Big|^2~\mathrm{d}x+ \sum^2_{j=1}\lambda_j \right \rbrack  \tag{$\mathfrak{P}$}
\end{align}
for~$u_d=\bar{u}-\bar{\varphi}$.
\end{prop}
\begin{proof}
First, note that the objective functional of~\eqref{probcounter} is convex in~$(\lambda_1,\lambda_2,\ell)$.
Calculating its partial derivatives at~$(\bar{\lambda}_1,\bar{\lambda}_2,0)$ reveals 
\begin{align*}
(u_j, \bar{u}-u_d)_{L^2}+1=(u_j, \bar{\varphi})_{L^2}+1=0 \quad \text{for all}~ j=1,2
\end{align*}
as well as
\begin{align}
(\ell, \bar{u}-f)_{L^2}=(\ell, \bar{\varphi})_{L^2}=0 \quad \text{for all}~\ell \in \Lc.
\end{align}
Consequently,~$(\bar{\lambda}_1,\bar{\lambda}_2,0)$ is a stationary point of \eqref{probcounter} and thus a minimizer.
\end{proof}
In particular, if~$\bar{\varphi}$ exists we have
\begin{align*}
\min_{\lambda \in \R^2_+,\ell \in \Lc} \left \lbrack \frac{1}{2} \int^{10}_0 \Big|\sum^2_{j=1}\lambda_j u_j(x)-u_d(x)\Big|^2~\mathrm{d}x+ \TGV \left(\sum^2_{j=1}\lambda_j u_j\right)\right \rbrack <  \min \eqref{probcounter}
\end{align*}
due to~$\TGV(\bar{u})< \bar{\lambda}_1+\bar{\lambda}_2$. In order to construct such a function, define
\begin{align*}
w_1(x)= \begin{cases}
1 & \bar{x}_1 \leq x \leq \bar{x}_2\\
0 & \text{else}
\end{cases},~
w_2(x)= \begin{cases}
1-\frac{3}{8}K_{\bar{x}_1}(x)  & \bar{x}_1 \leq x \leq \bar{x}_2\\
0 & \text{else}
\end{cases},~\ell_1(x)=1,~\ell_2(x)=x-5.
\end{align*}
Note that
\begin{align*}
(u_1,w_1)_{L^2} \neq 0,~(u_1,w_2)_{L^2}=0,~(u_2,w_2)_{L^2} \neq 0,~(u_2,w_1)_{L^2}=0
\end{align*}
as well as
\begin{align*}
(\ell_1,\ell_2)_{L^2} =0,~\operatorname{span}\{\ell_1,\ell_2\}=\Lc.
\end{align*}
Now, make the ansatz
\begin{align} \label{ansatzphi}
\varphi(\gamma_1,\gamma_2,\gamma_3,\gamma_4) \coloneqq \gamma_1 w_1+ \gamma_2 w_2+ \gamma_3 \ell_1+ \gamma_4 \ell_4.
\end{align}
The following proposition yields the desired counterexample.
\begin{prop}
There is~$(\bar{\gamma}_1,\bar{\gamma}_2,\bar{\gamma}_3,\bar{\gamma}_4)$ such that~$\bar{\varphi}\coloneqq \varphi (\bar{\gamma}_1,\bar{\gamma}_2,\bar{\gamma}_3,\bar{\gamma}_4)$ satisfies~\eqref{eq:condonphi}.
\end{prop}
\begin{proof}
Inserting the ansatz~\eqref{ansatzphi} into~\eqref{eq:condonphi} yields the linear system
\begin{align*}
\begin{pmatrix}
  (u_1,w_1)_{L^2} & 0 & (u_1,\ell_1)_{L^2} & (u_1,\ell_2)_{L^2} \\
0 &(u_2,w_2)_{L^2} & (u_2,\ell_1)_{L^2} & (u_2,\ell_2)_{L^2} \\
(\ell_1,w_1)_{L^2}  & (\ell_1,w_2)_{L^2}  & (\ell_1,\ell_1)_{L^2} & 0  \\(\ell_2,w_1)_{L^2}  & (\ell_2,w_2)_{L^2}  & 0& (\ell_2,\ell_2)_{L^2} 
\end{pmatrix}
\begin{pmatrix}
\gamma_1 \\
\gamma_2 \\
\gamma_3  \\
\gamma_4 
\end{pmatrix}=
\begin{pmatrix}
1\\
1 \\
0  \\
0 
\end{pmatrix}
\end{align*}
Explicitly calculating the entries of this matrix reveals its invertibility and thus the existence of a solution~$(\bar{\gamma}_1,\bar{\gamma}_2,\bar{\gamma}_3,\bar{\gamma}_4)$. The statement follows.
\end{proof}
\section*{Acknowledgments} We wish to thank Kristian Bredies for very interesting discussions and suggestions regarding the content of this work. The work of J.A.I. was partially supported by the State of Upper Austria.

\bibliographystyle{plain}
\bibliography{extreme_tgv}

\end{document}